\documentclass[a4paper]{amsart} 

\usepackage{latexsym,amsmath,amsfonts,amscd,amssymb}
\usepackage{lscape}
\usepackage{mathdots}
\usepackage{array}
\usepackage{pdfsync}
\usepackage[all]{xy}
\usepackage{amsthm}
\usepackage{multirow}
\usepackage[T1]{fontenc} % \tau gave problems
\usepackage{enumitem}
\usepackage{mathtools} % for matrices
\usepackage{mathdots} % for iddots

\usepackage[hidelinks]{hyperref}
\usepackage{dynkin-diagrams}

\usepackage{pifont} %% for cmark and xmark

\usepackage{tikz}
\usetikzlibrary{matrix}
\usetikzlibrary{cd} % to use tikzcd

\usepackage{mathscinet} % for cprime in references

\makeatletter 			%% to use [cc|c] in pmatrix environment
\renewcommand*\env@matrix[1][*\c@MaxMatrixCols c]{%
	\hskip -\arraycolsep
	\let\@ifnextchar\new@ifnextchar
	\array{#1}}
\makeatother

\usepackage{relsize}
\usepackage{afterpage}
\newcommand{\st}{\;|\;}

\newcommand{\cS}{\mathcal{S}}

\newcommand{\Z}{{\mathbb{Z}}}
\newcommand{\Q}{{\mathbb{Q}}}
\newcommand{\R}{{\mathbb{R}}}
\newcommand{\C}{{\mathbb{C}}}
\newcommand{\Oc}{{\mathbb{O}}}

\newcommand{\cA}{\mathcal{A}}

\newcommand{\cJ}{\mathcal{J}}

\newcommand{\cO}{\mathcal{O}}
\newcommand{\cQ}{\mathcal{Q}}
\newcommand{\cR}{\mathcal{R}}

\newcommand{\fa}{\mathfrak{a}}

\newcommand{\fe}{\mathfrak{e}}

\newcommand{\fg}{\mathfrak{g}}
\newcommand{\fh}{\mathfrak{h}}

\newcommand{\fp}{\mathfrak{p}}

\newcommand{\ft}{\mathfrak{t}}

\newcommand{\fso}{\mathfrak{so}}

\newcommand{\fB}{\mathfrak{B}}
\newcommand{\fJ}{\mathfrak{J}}

\newcommand{\al}{\alpha}
\newcommand{\be}{\beta}

\newcommand{\la}{\langle}
\newcommand{\ra}{\rangle}

\newcommand{\Cl}{\mathrm{Cl}}

\newcommand{\PSO}{\mathrm{PSO}}

\newcommand{\GL}{\mathrm{GL}}
\newcommand{\SL}{\mathrm{SL}}
\newcommand{\SO}{\mathrm{SO}}
\newcommand{\Sp}{\mathrm{Sp}}
\newcommand{\SSS}{\mathrm{S}}
\newcommand{\Spin}{\mathrm{Spin}}

\newcommand{\OO}{\mathrm{O}}
\newcommand{\HH}{\mathbb{H}}

\DeclareMathOperator{\Ad}{Ad}

\DeclareMathOperator{\rk}{rk}

\DeclareMathOperator{\Hom}{Hom}

\DeclareMathOperator{\End}{End}

\DeclareMathOperator{\Id}{Id}

\DeclareMathOperator{\Aut}{Aut}

\DeclareMathOperator{\Iso}{Iso}

\newcommand{\X}{}

\newcommand{\cmark}{\ding{51}}%
\newcommand{\xmark}{\ding{55}}%

\theoremstyle{plain}
\newtheorem{theorem}{Theorem}[section]

\newtheorem{lemma}[theorem]{Lemma}
\newtheorem{proposition}[theorem]{Proposition}
\newtheorem{example}[theorem]{Example}

\theoremstyle{definition}
\newtheorem{definition}[theorem]{Definition}

\newtheorem{innercustomthm}{Theorem}
\newenvironment{customthm}[1]
{\renewcommand\theinnercustomthm{#1}\innercustomthm}
{\endinnercustomthm}
\newtheorem{innercustomdef}{Definition}
\newenvironment{customdef}[1]
{\renewcommand\theinnercustomdef{#1}\innercustomdef}
{\endinnercustomthm}
\newtheorem{innercustomprop}{Proposition}
\newenvironment{customprop}[1]
{\renewcommand\theinnercustomprop{#1}\innercustomprop}
{\endinnercustomthm}
\theoremstyle{remark}

\newtheorem{remark}[theorem]{Remark}

\title[On the Gelfand property for complex symmetric pairs]{On the Gelfand property\\ for complex symmetric pairs}
\author{Roberto Rubio}
\address{Weizmann Institute of Science, Rehovot 76100, Israel}\address{Universitat Aut\`onoma de Barcelona, Bellaterra, Barcelona 08193, Spain}
\email{roberto.rubio@uab.es}

\begin{document}

%\tableofcontents
\begin{abstract}
	We first prove, for pairs consisting of a simply connected complex reductive group together with a connected subgroup, the equivalence between two different notions of Gelfand pairs. This partially answers a question posed by Gross, and allows us to use a criterion due to Aizenbud and Gourevitch, and based on Gelfand-Kazhdan's theorem, to study the Gelfand property for complex symmetric pairs. This criterion relies on the regularity of the pair and its descendants. We introduce the concept of a pleasant pair, as a means to prove regularity, and study, by recalling the classification theorem, the pleasantness of all complex symmetric pairs. On the other hand, we prove a method to compute all the descendants of a complex symmetric pair by using the extended Satake diagram, which we apply  to all pairs. 	
	 Finally, as an application, we prove that eight out of the twelve exceptional complex symmetric pairs, together with the infinite family $(\Spin_{4q+2}, \Spin_{4q+1})$, satisfy the Gelfand property, and state, in terms of the regularity of certain symmetric pairs, a sufficient condition for a conjecture by van Dijk and a reduction of a conjecture by Aizenbud and Gourevitch. 
\end{abstract}

\maketitle

\section{Introduction}

Given an irreducible representation $V$ of a group $G$ and a subgroup $H\subseteq G$, one very often asks how the representation $V$ decomposes, or branches, into irreducible representations of $H$. If, under suitable hypothesis, the trivial representation of $H$ appears at most once, then  $(G,H)$ is called a Gelfand pair. The applications of Gelfand pairs range from harmonic analysis \cite[\S 6]{dieudonne} to number theory \cite{gross}.

In this work we focus on a general definition of a Gelfand pair (Definition \ref{def:GP}), where $G$ and $H$ are reductive algebraic groups, and certain not necessarily finite-dimensional or unitary representations (the class of Casselman-Wallach representations) are considered. Other weaker definitions combine the representation with its contragredient (which we will refer to as Gelfand pairs \`a la Gross, Definition \ref{def:GP-Gross}) or deal with unitary representations (which we will call unitary Gelfand pairs, Definition \ref{def:GP-uni}). 

From the work of Harish-Chandra, characters for infinite-dimensional representations are defined as linear functionals on rapidly decreasing, or Schwartz, functions (see, for instance, \cite[Ch. 8]{wallach}). The theory of distributions plays an important role for Gelfand pairs too. By using relative characters, distributional criteria were proved for the weaker versions of the Gelfand property, as pioneered by Gelfand and Kazhdan in \cite{gelfand-kazhdan}, leading to a criterion for Gelfand pairs \`a la Gross in \cite{gross} for non-archimedean fields, whereas for unitary representations in archimedean fields, it is due to Thomas \cite{thomas}. A general version of this criterion for archimedean fields was later stated by Sun and Zhu \cite{sun-zhu}. 

A natural candidate for Gelfand pairs are symmetric pairs $(G,H)$, where $H\subseteq G$ is the subgroup of fixed points of an involution $\theta:G\to G$. Actually, it was conjectured by van Dijk that all complex symmetric pairs are unitary Gelfand pairs \cite[Conj. 2]{vanDijk-08}. In this direction, Aizenbud and Gourevitch used algebro-geometric techniques to provide a generalization of Harish-Chandra descent and turn Gelfand-Kazhdan's theorem into a more easily computable criterion \cite{ag-duke} for symmetric pairs that are moreover stable\footnote{or good, in the notation of \cite{ag-duke}.} (closed $H\times H$-orbits in $G$ are preserved by the anti-involution $g\mapsto \theta(g)^{-1}$). Namely: if a stable pair $(G,H)$ and its descendants (centralizers of certain semisimple elements) are regular ($H$-invariant distributions supported on the nilpotent cone are also invariant under the action of admissible elements), then, provided the existence of a certain anti-involution, it is a Gelfand pair. As an application, this criterion, together with a case-by-case proof of the equivalence between Gelfand pairs and Gelfand pairs \`a la Gross, was used to show that, for any local field $F$, the pairs $(\GL_{n+k}(F), \GL_n(F) \times \GL_k(F))$ and $(\GL_n(E), \GL_n(F))$, for $E$ a quadratic extension of $F$, are Gelfand pairs. Furthermore, the same was done in \cite{ag-transactions} to prove that the complex pairs $(\GL_n,\OO_n)$ and  $(\OO_{n+m},\OO_n \times \OO_m)$ are Gelfand pairs (it had already been proved that $(\SO_n,\SO_{n-1})$ is a unitary Gelfand pair \cite{aparicio-vanDijk} and a Gelfand pair \cite{ags-09}). It was conjectured by Aizenbud and Gourevitch \cite[Conj. 4]{ag-duke} that all symmetric pairs are regular, which would in particular imply van Dijk's conjecture, as all complex symmetric pairs are stable. Further evidence for these conjectures are the Gelfand property of the complex pair $(\GL_{2n},\Sp_{2n})$ proved in \cite{sayag}, and the regularity of nice symmetric pairs proved in \cite{aizenbud-13}. However, there have not been any other further progress, and results on exceptional symmetric pairs are scarce: the niceness, and hence regularity, of six of them was deduced in \cite{aizenbud-13}, and even in the real case, only the rank-one real symmetric pairs $(F_4^{-20},\Spin(1,8))$ and $(F_4^4,\Spin(4,5))$ were shown to be unitary Gelfand pairs \cite{vanDijk-86}.
 
The present work offers group-theoretic techniques to advance on these conjectures by proving an equivalence of two notions of a Gelfand pair, defining the notion of a pleasant pair and using it to prove the regularity of many symmetric pairs, and introducing and applying a method to compute descendants via the extended Satake diagram. As an application, eight exceptional pairs and an infinite $\Spin$ family are proved to be Gelfand pairs, and the conjectures are reduced to a statement about the regularity of certain pairs.

First, we prove an equivalence of the notions of Gelfand pair and Gelfand pair \`a la Gross for certain pairs of connected complex groups.
\begin{customthm}{\ref{theo:eqGP1-GP2}}
	Let $(G,H)$ be a pair of complex connected reductive groups with $H\subseteq G$ and $G$ simply connected. The pair $(G,H)$ is a Gelfand pair if and only if it is a Gelfand pair \`a la Gross.
\end{customthm}
{\noindent This answers a question posed by Gross \cite[\S 4]{gross} for the groups relevant for our work. It moreover suggests that we can systematically use the Aizenbud-Gourevitch criterion to prove the Gelfand property of complex symmetric pairs $(G,H)$. Namely, if a stable (Definition \ref{def:stable}) symmetric pair $(G,H)$ and all its descendants (Definition \ref{def:descendant}) are regular (Definition \ref{def:regular}), then $(G,H)$ is a Gelfand pair. In order to use it, two important observations are made:}
\begin{itemize}
	\item the fixed-point subgroup by $\theta$ of a simply connected group is connected and any complex symmetric pair is stable, so the Aizenbud-Gourevitch criterion can be used on symmetric $(G,H)$ with $G$ simply-connected to prove the Gelfand property (Proposition \ref{prop:criterion-complex}).
	\item the Gelfand property for symmetric $(G,H)$ with $G$ simply-connected implies the Gelfand property for the symmetric pairs corresponding to quotients of $G$ by finite central subgroups (Lemma \ref{lem:quot-sym-pair}). 
\end{itemize}

Secondly, we formally introduce and study the notion of a pleasant pair as a way to prove regularity. Set $\cA_\theta=\{ g \in G \st \theta(g)\in g Z(G) \}$, with $Z(G)$ the centre of~$G$.
\begin{customdef}{\ref{def:pleasant}}
	We say that a pair $(G,H,\theta)$ is \textbf{pleasant} when $$ \Ad \cA_\theta \subseteq \Ad H,$$
or, equivalently written in terms of the involution $\theta$, when,
$$ (\Ad G)^\theta \subseteq  \Ad G^\theta.$$
\end{customdef} {\noindent Here we denote by $\Ad:G\to \Aut G$ the action by conjugation.  By the definition of regularity, a pleasant pair is immediately regular. We then study the pleasantness of all complex symmetric pairs by means of the classification theorem. Our results are summed up in Tables \ref{tab:pleasant-nice}, \ref{tab:pleasant-nice2} and \ref{tab:pleasant-nice3}, where it can be seen that the concept of pleasant pair is especially helpful for exceptional pairs and certain $\Spin$ pairs. For convenience we use the notation of the Lie types even for the classical cases (this notation can be read from Table \ref{tab:cla-des}), and use abbreviations like $BD$ for orthogonal and Spin groups (Lie types $B$ and $D$).}
 
Thirdly, we consider the extended Satake diagram and use it to describe the  descendants of symmetric pairs, where $P$ denotes $\{ g\theta(g)^{-1} \st g\in G\}$.
\begin{customthm}{\ref{theo:visual-descendants}}
	The Satake diagrams of descendants $(G_x,H_x,\theta_{|G_x})$ for semisimple $x\in P$ are (possibly disconnected) Satake diagrams obtained by erasing at least one white node (and its incident edges) from the extended Satake diagram of $(G,H,\theta)$ in such a way that nodes connected by a bar are both kept or erased.
\end{customthm}
 {\noindent All descendants are actually computed in Tables \ref{tab:cla-des} and \ref{tab:exc-des}.}
  
 Finally, we combine all these results to deduce the Gelfand property for eight exceptional symmetric pairs, as well as for the infinite family $(\Spin_{4q+2},\Spin_{4q+1})$.
\begin{customthm}{\ref{theo:8-Gelfand}}
	The complex symmetric pairs $(G_2,A_1+A_1)$, $(F_4,B_4)$, $(F_4,C_3+A_1)$, $(E_{6},F_4)$, $(E_{6},C_4)$, $(E_6,F_4)$, $(E_6,A_5+A_1)$, $(E_{7},A_7)$ and $(E_8,D_8)$, together with the infinite family $(\Spin_{4q+2},\Spin_{4q+1})$, are Gelfand pairs.
\end{customthm}

{\noindent The proof for the remaining exceptional complex symmetric pairs is reduced to the proof of regularity of just one exceptional pair and some low-rank classical ones.}
\begin{customprop}{\ref{prop:exceptional-Gelfand}}
	All exceptional complex symmetric pairs are Gelfand pairs if $(D_4,B_3)$, $(D_4,A_3+\C)$,  $(D_5,D_4+\C)$, $(D_6,A_{5}+\C)$, $(D_7,B_5+B_1)$,  $(E_7,E_6+\C)$ and $(D_8,D_6+D_2)$ are regular. 
\end{customprop}

{\noindent Moreover, we give, in terms of the regularity of some pairs, a sufficient condition for van Dijk's conjecture, and a reduction of Aizenbud-Gourevitch conjecture.}
\begin{customprop}{\ref{prop:conjectures}}
	All complex symmetric pairs are regular (Aizenbud-Gourevitch conjecture) and Gelfand pairs (implying van Dijk's conjecture) if the families of pairs $(D_r,A_{r-1}+\C)$, $(C_{2r},C_r+C_r)$, the families $(\Spin_{r+s},\Spin_r \times_{\Z_2} \Spin_s)$ for $|r-s|>2$, and $r,s$ even if $r+s\neq 4q+2$,, and the pair $(E_7,E_6+\C)$ are regular. 
\end{customprop}

{\noindent We note here that the pairs $(D_r,A_{r-1}+\C)$ and  $(C_{2r},C_r+C_r)$ (corresponding to the classical pairs $(\SO_{2r},\GL_r)$ and $(\Sp_{4r},\Sp_{2r}\times \Sp_{2r})$) were already well known to be challenging cases for the Gelfand property. Our work highlights the crucial importance of looking also at $\Spin$ pairs and not just orthogonal ones.

	 An answer to the remaining part of van Dijk's and Aizenbud-Gourevitch conjectures will require different techniques (perhaps finer distributional criteria than those in \cite{ag-duke}) in order to prove the regularity of the pairs in Proposition \ref{prop:conjectures}, and is a long-term joint project with Carmeli. The study of other fields (starting with the equivalence of Gelfand pair notions and the computation of descendants) is also a natural continuation of the present work.

The paper is structured as follows. In Section 2, the different notions of Gelfand pairs are introduced and the equivalence between two of them is proved. In Section 3, the Aizenbud-Gourevitch criterion, together with the  notions of descendants and regularitiy, are recalled, and we make and prove the necessary considerations in order to apply the criterion to our case. Section 4 starts with the introduction of the notion of pleasantness, together with several criteria to prove it. We then recall the classification of complex symmetric pairs by means of Satake diagrams (Section \ref{sec:Satake}) and use it to study the pleasantness of all pairs (Sections \ref{sec:classical-pleasant}, \ref{sec:spin-pairs} and \ref{sec:exceptional-pleasant}). This includes recalling concrete realizations of Spin (Section \ref{sec:spin-pairs}) and exceptional (Section \ref{sec:exceptional-lie-groups}) groups. The results, together with a list of nice pairs, are summed up in Section \ref{sec:nice-summary}. In Section \ref{sec:visual} we review the structure of centralizers of semisimple elements and prove a result about the computation of descendants of symmetric pairs, which is applied to all pairs (Tables \ref{tab:cla-des} and \ref{tab:exc-des}), with special attention to exceptional and Spin pairs (Section \ref{sec:exceptional-spin-descendants}). Finally, Section \ref{sec:gelfand-pairs-conjectures} contains the applications of our previous results to prove the Gelfand property for exceptional and Spin pairs, together with a sufficient condition for van Dijk's conjecture and a reduction of Aizenbud-Gourevitch conjecture.

\textbf{Acknowledgments:} I would like to express my gratitude to Avraham Aizenbud and Dmitry Gourevitch for introducing me to this subject through many insightful discussions. Thanks also to Shachar Carmeli, Demetris Deriziotis, Gerrit van Dijk, James Humphreys and Itay Glazer for pointing out references or helpful discussions, and to Siddhartha Sahi and Eitan Sayag  for their interest in this project. Additionally, D. Gourevitch and S. Carmeli read a first version of this manuscript, for which I am also grateful.

This work has been possible thanks to the ERC StG grant 637912, ISF
grant 687/13, a grant from the Minerva foundation, and the MSCA project 750885 GENERALIZED.

\section{The notion of a Gelfand pair and a theorem of equivalence}

This section recalls several notions of a Gelfand pair and establishes an equivalence between two of them. This equivalence is crucial to our work and answers, in the complex case, a question posed by Gross \cite[\S 4]{gross}. 

\subsection{Various Gelfand properties}

We work over the complex numbers, although the definitions of this section are valid for an archimedean field and are easily adapted to the non-archimedean case.

Let $G$ be a reductive algebraic group and $H\subseteq G$ a subgroup. We consider the class of irreducible Casselman-Wallach representations (see  \cite{sun-zhu} and \cite{ags-compositio}).

\begin{definition}
	An irreducible representation $E$ of $G$ is said to be irreducible Casselman-Wallach when it is 
	\begin{itemize}
		\item Fr\'echet, that is, $E$ is a Fr\'echet vector space and the action of $G$ is continous.
		\item smooth, that is, for each $v\in E$, the map $g\mapsto \pi(g)(v)$ is smooth.% and for any $X\in\fg$, the differentiation map $d\pi(X):E\to E$ is continuous.
		\item of moderate growth in the sense of \cite{casselman-89} (for every  seminorm  $\rho$  on  $E$  there  exist  a positive  integer $N$  and  seminorm  $\nu$  such  that  $||\pi(g)(v)||_{\rho}\leq ||g||^N ||v||_{\nu}$ for $v\in E$ and $g\in G$, with $||g||$ coming from a suitable embedding of $G$). 
				% , that is, the action of $G$ is continuous with respect to any seminorm $\rho_i$ from a family $\{\rho_i \}$  generating the topology of $E$.
		\item admissible, that is, the multiplicity on $E$ of any irreducible representation of a maximal compact subgroup $K\subset G$ is finite.
		% underlying $(\fg,K)$-module is admissible, that is, 
	\end{itemize}
\end{definition}
\begin{remark}
	In general, one should also ask for $Z(\fg^\C)$-finiteness (a finite codimensional ideal of $Z(\fg^\C)\subset U(\fg^\C)$, with $U(\fg^\C)$ the universal envelopping algebra, annihilates $E$). This property is automatically satisfied for irreducible representations. 
\end{remark}

We then have the following definition.

\begin{definition}\label{def:GP}
	We say that $(G,H)$ is a Gelfand pair if any irreducible Casselman-Wallach representation $(\pi,E)$ of $G$ is multiplicity free for the trivial representation of $H$. Namely,
	$$ \dim \Hom_H(E,\C)\leq 1.$$ 
\end{definition}

An a priori weaker Gelfand property was introduced by Gross, without being ``able to establish the equality in the general case'', and because it was ``the one that arises most naturally''. It is indeed this property the one that admits the Aizenbud-Gourevitch criterion we shall introduce in Section \ref{sec:Aizenbud-Gourevitch}.  Let $E'\subseteq E^*$ denote the contragredient representation, which consists of the smooth linear forms.

\begin{definition}\label{def:GP-Gross}
	We say that $(G,H)$ is a Gelfand pair \`a la Gross if any irreducible Casselman-Wallach representation $(\pi,E)$ of $G$ satisfies
	$$ \dim \Hom_H(E,\C)\cdot \dim \Hom_H(E',\C) \leq 1.$$ 
\end{definition}

There is yet another commonly used Gelfand property in the literature (\cite{thomas}, \cite{vanDijk-86}), which will not be used directly in this work, but it is important to understand the state of the art.
	
\begin{definition}\label{def:GP-uni}
	We say that $(G,H)$ is a unitary Gelfand pair if any irreducible unitarizable admissible representation on a Hilbert space $(\pi,E)$ satisfies
	$$ \dim \Hom_H(E^{\infty},\C)\leq 1.$$ 
\end{definition}	
	
	Any Gelfand pair is a Gelfand pair \`a la Gross, and any Gelfand pair \`a la Gross is a unitary Gelfand pair.
	
\begin{remark}\label{rem:GP1-2-3}
The notion in Definition \ref{def:GP-uni} was referred to as a generalized Gelfand pair (generalized in the sense that $H$ need not be compact) but we find unitary Gelfand pair more descriptive nowadays. In \cite{ags-compositio}, Definition \ref{def:GP} was labelled as GP1, whereas Definition \ref{def:GP-Gross} was labelled as GP2 and Definition \ref{def:GP-uni} as GP3.
\end{remark}

\subsection{A theorem about the equivalence of two notions}
\label{sec:two-notions}

In this section we prove that Definitions \ref{def:GP} and \ref{def:GP-Gross} are equivalent over the field of complex numbers when $H$ is reductive. In the notation of Remark \ref{rem:GP1-2-3}, we prove that GP1 is equivalent to GP2. For this we need to recall the notion and a theorem about admissible morphisms.

\begin{definition}[\cite{ag-duke}, Def. 6.0.1]\label{def:admissible}
	Let $\pi$ be an action of a reductive group $G$ on a smooth affine variety $X$. We say that an algebraic automorphism $\sigma$ of $X$ is $G$-admissible if it
	\begin{itemize}
		\item normalizes $\pi(G)$, that is, $\sigma \pi(G) \sigma^{-1} \subseteq \pi(G)$.
		\item squares to an element of the action, that is, $\sigma^2\in\pi(G)$.
		\item  preserves closed $G$-orbits $\cO\subset X$, that is,  $\sigma(\cO)=\cO$.	
	\end{itemize}
\end{definition}

\begin{theorem}[\cite{ag-duke}, Thm. 8.2.1]\label{thm:ag-duke-821}
	Let $G$ be a reductive group, and let $\sigma$ be an $(\Ad G)$-admissible antiautomorphism of $G$. Let $\theta$ be the automorphism of G defined by $\theta(g) := \sigma(g^{-1})$. Let $(\pi,E)$ be an irreducible Casselman-Wallach representation of $G$. Then $E'\simeq E^\theta$, where 
	$E^\theta$ is $E$ twisted by $\theta$. 
\end{theorem} 

Consequently, if $H\subset G$ is a reductive subgroup, and we have an $(\Ad G)$-admissible antiautomorphism $\sigma$ of $G$ such that $\sigma(H) = H$, then 
$$\dim \Hom_H(E,\C)\cdot \dim \Hom_H(E',\C)=(\dim \Hom_H(E,\C))^2,$$
and the equivalence between Definitions \ref{def:GP} and \ref{def:GP-Gross} follows.

\begin{example}
	For $(\SL_n,\SO_n)$ or $(\SL_{p+q},\SSS(\GL_p\times \GL_q))$, the transposition map gives the equivalence between the notions of Gelfand pair and Gelfand pair \`a la Gross. The corresponding automorphism $\theta$ is the inverse-transpose map.
\end{example}

The following theorem establishes the equivalence between the notions of Gelfand pair and Gelfand pair \`a la Gross for complex reductive groups.

\begin{theorem}\label{theo:eqGP1-GP2}
	Let $(G,H)$ be a pair of complex connected reductive groups with $H\subseteq G$ and $G$ simply connected. The pair $(G,H)$ is a Gelfand pair if and only if it is a Gelfand pair \`a la Gross. In other words, GP1 is equivalent to GP2.
\end{theorem}

\begin{proof}

	Let $T_H$ be a maximal torus of $H$. Take any maximal torus $T\subset G$ containing $T_H$, consider a Chevalley basis $\{X_\al\}\subset \fg$  (that is, such that $\alpha([X_\al,X_{-\al}])=2$) with respect to $T^{ss}:=T\cap G^{ss}$, where `$ss$' denotes the semisimple part,  and let $\sigma$ be the antiautomorphism defined by $\sigma(t)=t$ for $t\in T$ and 
	\begin{equation}\label{eq:def-sigma}
	d\sigma(X_{\al})=-X_{-\al}.
	\end{equation}
	Such an antiautomorphism exists as $G$ is simply connected.
	
	We first check that $\sigma$ is $(\Ad G)$-admissible for the action $\Ad$ of $G$ on itself by conjugation. We see that $\sigma$ normalizes the conjugation $\Ad g$ by $g\in G$. Indeed, for $x\in G$,
	$$ (\sigma \circ (\Ad g)\circ \sigma^{-1})(x)= \sigma(g\sigma(x)g^{-1}) = (\Ad {\sigma(g^{-1})})(x).$$ 
	We have that $\sigma$ is an involution, so $\sigma^2=\Id\in \Ad G$. On the other hand,  the orbits are the conjugacy classes, and the closed orbits are the semisimple orbits. The action of $\sigma$ does preserve the orbits, as $\sigma$ is the identity on $T$.
	
	Secondly, we see that $\sigma(H)=H$. As $H$ is connected, it suffices to check, for $\fh$ the Lie algebra of $H$, that $d\sigma(\fh)=\fh$. Consider a Chevalley basis $\{Y_\be\}\subset \fh$ with respect  to $T_H^{ss}:=T_H\cap H^{ss}$. Trivially, $\sigma(T_H)=T_H$, so it suffices to show that $d\sigma(Y_\be)$ belongs to $\fh$. The main issue here is that a root vector $Y_\be$ is not necessarily one of the $X_\al$. This is only the case for regular subalgebras, but not in special subalgebras (see, e.g., \cite{dynkin}).
	
	We express $Y_\beta$ in terms of the Lie algebra $\ft$ of the torus $T$ and the root vectors of~$\fg$,
	$$ Y_\beta = Z + \sum_i \lambda_i X_{\al_i},$$
	with $Z\in \ft$. By $[U,Y_\beta]=\beta(U) Y_\beta$ for $U\in T_H^{ss}$, we have that $Z=0$ and $\al_i(U)=\be(U)$ for $U\in T_H^{ss}$. In other words, $Y_\be$ is a linear combination of the $X_{\al}$ such that the restriction of $\al$ to $T_H^{ss}$ is $\be$. 
	
	We prove now that $d\sigma(Y_\beta)=Y_{-\beta}$, and hence $Y_\beta\in \fh$, as $H$ is reductive ($\beta$ is a root if and only if $-\beta$ is a root). From
	$$ d\sigma (Y_\beta) = d\sigma(\sum_i \lambda_i X_{\al_i} ) = - \sum_i \lambda_i X_{-\al_i},$$
	we have, for $U\in T_H^{ss}$,
	\begin{align*}
	[U, d\sigma(Y_\beta)]  & = [U, - \sum_i \lambda_i X_{-\al_i}] = -\sum_i \lambda_i \al_i(U) X_{-\al_i} =  \beta(U) \sum_i \lambda_i X_{-\al_i}\\ & = -\beta(U) d\sigma(Y_\beta),
	\end{align*} 
	so $d\sigma(Y_\beta)=Y_{-\beta}\in \fh$ and consequently $\sigma(H)=H$ and, by Theorem \ref{thm:ag-duke-821} we have the equivalence.
\end{proof}

\begin{remark}
	Both the transposition map and $g\mapsto \sigma(g)$ in the proof of Theorem \ref{theo:eqGP1-GP2} are instances of Chevalley involutions (involutions inverting the elements of a maximal torus, see, e.g., \cite[\S 2]{adams-vogan}) composed with the inversion. Chevalley involutions are not uniquely defined, and the choice of a maximal torus containing a maximal torus of $H$ and the condition \eqref{eq:def-sigma} were crucial for the proof.
% differ by $N_T(H)$.
\end{remark}

As it is out of the scope of this paper, we leave for future work the proof of this equivalence for other fields. This would require a good understanding of the Satake-Tits classification and an analogue of the Chevalley involution for non-algebraically closed fields (as in \cite{adams-14} for the real case).

\section{The Aizenbud-Gourevitch criterion for symmetric pairs}
\label{sec:Aizenbud-Gourevitch}

Once we have settled the equivalence between the general definition of a Gelfand pair and the one given by Gross, we recall in the first two sections a criterion that will allow us to prove the Gelfand property for symmetric pairs. In Section \ref{sec:covering-quotient} we make some further considerations regarding products and quotients of pairs. % that will nevertheless be important.

% , following \cite{ag-duke}

\subsection{Distributional criterion}

 We denote the space of Schwartz distributions by $\cS^*(M)$ for $M$ a smooth complex algebraic variety.  When $M$ is affine, this is the space of linear functionals on Schwartz functions $\cS(M)$ or rapidly decreasing (together with all its derivatives) functions. Again, more general setups include Nash manifolds and non-archimedean fields. However, in this work we will only refer to $M$ being a group $G$, the vector space $\fp$, its regular elements $\cR \fp$ or the quotient space $\cQ\fp$. Moreover, we will not deal directly with distributions. See \cite{ag-duke} for more details on Schwartz functions and distributions.
 
% The action of $H\times H$
 
A distributional criterion to prove the Gelfand property for unitary representations (see Remark \ref{rem:GP1-2-3}) was given by Thomas \cite[Thm. E]{thomas}  and used by van Dijk \cite[Crit. 1.2]{vanDijk-86}. This was combined with Gelfand and Kazhdan's work \cite{gelfand-kazhdan} to give a more general criterion.

\begin{theorem}[\cite{ags-compositio}, Thm. 2.3.2]\label{theo:GK-ags}
Let $(G,H)$ be a pair of a reductive group $G$ and a subgroup $H\subseteq G$. Let $\tau$ be an involutive antiautomorphism of $G$ such that $\tau(H)=H$. Suppose that
$$\cS^*(G)^{H\times H}\subseteq \cS^*(G)^\tau.$$
Then, $(G,H)$ is a Gelfand pair in the sense of Gross.
\end{theorem}

From now on we will be interested in symmetric pairs.

\begin{definition}
	A pair of groups $(G,H)$ is called symmetric if there exists an automorphism $\theta:G\to G$ such that $H=G^{\theta}$. 
\end{definition}

Every symmetric pair comes with the antiautomorphism 
$$\sigma: g\mapsto \theta(g^{-1}).$$
When talking about a symmetric pair $(G,H)$, the maps $\theta$ and $\sigma$ will denote the involution and antiinvolution just described. We will sometimes denote the pair by $(G,H,\theta)$, especially when we use $\theta$ or $\sigma$.

If the hypotheses of Theorem \ref{theo:GK-ags} are satisfied for a symmetric pair with the antiautomorphism $\sigma$, the pair is called a GK pair. 

\begin{definition}[\cite{ag-duke}, Def. 7.1.8]
	A symmetric pair $(G,H,\theta)$ is called a GK-pair if all $H\times H$-invariant Schwartz distributions on $G$ are moreover $\sigma$-invariant.
\end{definition}

Thus, a GK-pair is a Gelfand pair \`a la Gross, by applying Theorem \ref{theo:GK-ags} with $\sigma=\tau$. The notion of a GK-pair admits the application of a generalization of the Harish-Chandra descent which shall result in Theorem \ref{theo:aizenbud-gourevitch}.

\subsection{Regularity, descendants, and the criterion}

We first recall some definitions from Section 7 of \cite{ag-duke}. 

\begin{definition}
	Let $(G,H, \theta)$ be a symmetric pair. An element $g\in G$ is called admissible if 
	\begin{itemize}
		\item the map $\Ad g$ commutes with $\theta$, that is, $g^{-1}\theta(g)\in Z(G)$. 
		\item the restriction $(\Ad g)_{|\fp}$ is $H$-admissible in the sense of Definition \ref{def:admissible}. 
	\end{itemize}
\end{definition}

Let $\cR \fp$ denote the complement of the nilpotent cone of $\fp$, and $\cQ\fp:=\fp/\fp^G$ the quotient of $\fp$ by its $\Ad G$-invariant elements.
	
%[\cite{ag-duke}, Def. 7.4.2] 	
\begin{definition}\label{def:regular}
A symmetric pair $(G,H,\theta)$ is called \textbf{regular} when for any admissible $g\in G$ such that $\cS^*(\cR\fp))^H \subseteq \cS^*(\cR\fp))^{\Ad g}$ we have 
\begin{equation}\label{eq:regularity}
\cS^*(\cQ\fp)^H \subseteq \cS^*(\cQ\fp)^{\Ad g}.
\end{equation}
\end{definition}

Define the subvariety
$$ P= \{ g\theta(g)^{-1} \st g\in G \}.$$
We give an a priori weaker definition of descendant (compared to \cite[Def. 7.2.2]{ag-duke}), which is more suitable in our setup.

\begin{definition}\label{def:descendant}
	A \textbf{descendant} of a symmetric pair $(G,H,\theta)$ is a symmetric pair
	$$(G_x,H_x,\theta_{|G_x})$$
	for some semisimple $x\in P$, where $G_x$ and $H_x$ denote the centralizer of $x$.
\end{definition}
Note that the map $\theta$ restricts to an automorphism of~$G_x$ since $x\in P$.

\begin{remark}
	The fact that Definition \ref{def:descendant} is an a priori weaker notion follows from \cite[Prop. 7.2.1(i)]{ag-duke} and makes possible a simpler application of Theorem \ref{theo:aizenbud-gourevitch} and Proposition \ref{prop:criterion-complex} below (these results are of course also true in the generality of \cite[Def. 7.2.2]{ag-duke}).
\end{remark}

%\subsection{The criterion}

\begin{definition}\label{def:stable}
	A symmetric pair $(G,H, \theta)$ is called stable if, for any closed $(H \times H)$-orbit $\mathcal{O} \subset G$, we have $\sigma(\cO) = \cO$.
\end{definition}

Harish-Chandra descent (see, for instance, \cite{adams-vogan-92}) describes the behaviour of invariant distributions around a semisimple element $s$ by means of an invariant distribution on the centralizer of $s$ in a neighbourhood of the identity. A generalization of this principle was applied to the distributional criterion, Theorem \ref{theo:GK-ags}, to provide a way to prove the Gelfand property \cite[Thm. 7.4.5]{ag-duke}, which we rename after its authors.

\begin{theorem}[Aizenbud-Gourevitch criterion] \label{theo:aizenbud-gourevitch}
	Let $(G,H,\theta)$ be a stable symmetric pair such that all its descendants (including itself) are regular. Then $(G,H,\theta)$ is a GK-pair.
\end{theorem}

Combining this theorem with our development in Section \ref{sec:two-notions}, we can prove the following.

\begin{proposition}\label{prop:criterion-complex}
	Let $(G,H)$ be a complex symmetric pair with $G$ simply connected. If all its descendants (including itself) are regular, then $(G,H)$ is a Gelfand pair.
\end{proposition}

\begin{proof}
	First, we note that all complex symmetric pairs satisfy the stability condition \cite[Sec. 7.1]{ag-duke}, so by Theorem \ref{theo:eqGP1-GP2}, the pair $(G,H)$  is a Gelfand pair \'a la Gross. Since $G$ is simply connected, the subgroup of fixed points $G^\theta$ is connected (this is a fact that goes back to Borel \cite[Thm. 3.4, fn. 4]{borel-61}). Hence, we can apply Theorem \ref{theo:eqGP1-GP2} to deduce that  $(G,H)$ is a Gelfand pair.
\end{proof}

This criterion, together with a case by case proof of the equivalence of notions, was first used in \cite{ag-duke} to show that, for any local field $F$, the pairs $(\GL_{n+k}(F), \GL_n(F) \times \GL_k(F))$ and $(\GL_n(E), \GL_n(F))$, for $E$ a quadratic extension of $F$, are Gelfand pairs. Then, the same was done in  \cite{ag-transactions} to prove that the complex pairs $(\GL_n,\OO_n)$ and  $(\OO_{n+m},\OO_n \times \OO_m)$ are Gelfand pairs. 

In this work we shall introduce and study pleasantness as an extra criterion to prove regularity, describe the descendants in full generality, and apply these to prove the Gelfand property for many exceptional pairs. We shall see first that studying pairs $(G,H)$ with $G$ simply connected is not a restrictive hypothesis, but instead puts us in the most general situation (Lemma \ref{lem:quot-sym-pair}).

\subsection{Coverings, quotients and products of pairs}
\label{sec:covering-quotient}

We have defined the Gelfand property and regularity for reductive groups. We make here some considerations which will allow us to deduce information from the study of simple, and often simply connected, groups.

\begin{lemma}
	Let $(G,H)$, $(K,L)$ be Gelfand pairs and $F$ a normal finite subgroup of $G$:
\begin{itemize}
	\item the product $(G\times K,H\times L)$ is a Gelfand pair.
	\item the quotient $(G/F,HF/F)$ is a Gelfand pair. 
\end{itemize}

\end{lemma}

\begin{proof}
	An irreducible Casselman-Wallach representation $V$ of $G\times K$ is isomorphic to the completed tensor product $V_1\hat{\otimes} V_2$ where $V_1$, $V_2$ are irreducible Casselman-Wallach representations of, respectively, $G$ and $K$. We then have
	$$\Hom_{H\times L}(V,\C)\cong \Hom_H(V_1,\C)\hat{\otimes} \Hom_L(V_2,\C).$$
	%	If the trivial representation appears more than once on $V_{|H\times L}$, then the trivial representation of $H$ (or $L$) will appear more than once on $V_{|H}$ (or $V_{|L}$).
	
	For the second part, an irreducible Casselman-Wallach representation $V$ of $G/F$ can be regarded as an irreducible Casselman-Wallach (here we are using that $F$ is finite) representation $V_G$ of $G$, where $F$ acts trivially. Then, 
	$$\Hom_{HF/F}(V,\C) \cong \Hom_{H}(V_G,\C).$$
\end{proof}

Note that if $(G,H)$ is a symmetric pair and $F$ a normal finite subgroup, the quotient $(G/F,HF/F)$ is not necessarily a symmetric pair. Instead, we have the following definition.

\begin{definition}\label{def:quotient-sym-pair}
	A quotient symmetric pair of $(G,H,\theta)$ by a $\theta$-invariant finite normal subgroup $F\subset Z(G)$ is the symmetric pair $(G',H',\theta')$ where $G'=G/F$, the involution $\theta'$ is the one induced on $G'$ by the Lie algebra involution $d\theta:\fg\to \fg$ and $H'$ is the fixed-point subgroup $(G')^{\theta'}$. 
\end{definition}

Note that it is possible to lift $d\theta$ to $G'$ because it already lifts to $G$ and we know $\theta(F)=F$. The fixed-point subgroup $H'$ has the same Lie algebra as $H$, since the two pairs have the same underlying tuple $(\fg,\fh,d\theta)$. Moreover, as $\theta(hF)=\theta(h)F=hF$, we have that 
$$HF/F \subset H'.$$

\begin{example}\label{ex:SO-PSO}
	Consider the symmetric pair $(G,H)=(\SO(2n),\SSS(\OO(n)\times \OO(n)), \theta)$, where $\theta$ is conjugation by the $(n,n)$-block matrix
	$\begin{psmallmatrix}1&0\\0&-1\end{psmallmatrix}$. Set $$F=Z(\SO(2n))=\{1, \begin{psmallmatrix}-1&0\\0&-1\end{psmallmatrix} \}.$$ Consider the group $\PSO(2n) = \SO(2n) / Z(\SO(2n)).$ 	We have that $\PSO(2n)^\theta$ is 
	$$ \frac{HZ}{Z}\ltimes [\begin{psmallmatrix}0&-1\\1&0\end{psmallmatrix} ].$$
\end{example}

We are interested in the following lemma.

\begin{lemma}\label{lem:quot-sym-pair}
With the notation of Definition \ref{def:quotient-sym-pair}, if $(G,H,\theta)$ is a Gelfand pair, then $(G',H',\theta')$ is a Gelfand pair.	
\end{lemma}

\begin{proof}
	It follows from the inclusion $HF/F \subset H'$ and
	$$\dim \Hom_{H'}(V,\C)  \leq  \dim \Hom_{HF/F}(V,\C) = \dim \Hom_{H}(V_G,\C).$$
\end{proof}

In general, a reductive group $G$ can be written as
$$ G = Z(G)_o \times_F G^{ss}$$
for some finite group $F$, where $Z(G)_o$ is the identity component of the centre of $G$, and the semisimple part $G^{ss}$ itself is the quotient of a product of simple groups $\{G_i\}$ by some finite group. An involution on $G$ gives involutions on $Z(G)_o$ and on each simple factor $G_i$ (these involutions are compatible in some sense given by the quotients by finite groups). A sufficient condition for $(G,H,\theta)$  to be a Gelfand pair is that the symmetric pairs for $Z(G)_o$ and each simple factor $G_i$ are Gelfand pairs. In particular, we have the following.

\begin{proposition}
 The Gelfand property is satisfied for all complex symmetric pairs $(G,H,\theta)$ with $G$ semisimple and connected if and only if it is satisfied for all $G$ simple, connected, and simply connected.
\end{proposition}

We will therefore aim to prove that symmetric pairs for simply connected simple groups satisfy the Gelfand property. Note that for inner symmetric pairs, the involution on $Z(G)_o$ is always trivial.

\begin{example}
	If we want to consider the pair  $(\GL_{p+q},\GL_p\times \GL_q)$ from our viewpoint, we would consider first the pair $(\SL_{p+q},\SSS(\GL_p\times \GL_q))$, which will be proved to be a Gelfand pair. Since $\GL_{p+q}=\SL_{p+q} \times_{\Z_{p+q}} \C^*$ and the pair is inner, the well-known Gelfand property for $(\GL_{p+q},\GL_p\times \GL_q)$ would follow from the observation above and Lemma \ref{lem:quot-sym-pair}.
\end{example}

\section{Pleasant pairs and regularity}
\label{sec:pleasant}

We introduce and study the notion of pleasant pair as a group-theoretic way of proving the regularity of a symmetric pair. 

\subsection{Definition of a pleasant pair}

Condition \ref{eq:regularity} for regularity (Definition \ref{def:regular}) is trivially satisfied when $\Ad g \in \Ad H$ for all admissible $g$. Since the admissibility condition (Defintion \ref{def:admissible}) is fairly complicated, we define the set
$$ \cA_\theta:=\{ g \in G \st \theta(g)\in g Z(G) \} = \{ g \in G \st \sigma(g) g \in Z(G) \}, $$
which contains all admissible elements and is easily described. 

\begin{definition}\label{def:pleasant}
	We say that a pair $(G,H,\theta)$ is \textbf{pleasant} when $$ \Ad \cA_\theta \subseteq \Ad H,$$
or, equivalently written in terms of the involution $\theta$, when,
$$ (\Ad G)^\theta \subseteq  \Ad G^\theta.$$
\end{definition}

An immediate consequences of this definition is the following.

\begin{lemma}
	Pleasant pairs are regular.
\end{lemma}

\begin{example}\label{ex:regular-extreme}
		For $G$ a simple group, the symmetric pair $(G\times G,\Delta G)$, where $\Delta G$ denotes the diagonal, is always regular. Indeed, the involution is given by $\theta(x,y)=(y,x)$, so the elements in $\cA_\theta$ are just $(x,xz)$ with $z$ in the centre, and $\Ad (x,xz)=\Ad (x,x)\in \Ad \Delta G$. As an additional trivial remark, note that the trivial pair $(G,G)$ is also regular.
\end{example}

\begin{example}\label{ex:SL2-C}
	Consider the pair $(\SL_2,\C^*)$, given by the involution $\theta=\Ad {I_{1,1}}$, where $$I_{1,1}=\begin{psmallmatrix}1&0\\0&-1\end{psmallmatrix}.$$ The subgroup $\C^*$ corresponds to the diagonal matrices $\begin{psmallmatrix}a&0\\0&a^{-1}\end{psmallmatrix}$. We have $Z(\SL_2)=\la -\Id \ra \cong \Z_2$ and $$J=\begin{psmallmatrix}0&-1\\1&0\end{psmallmatrix}\in \cA_\theta,$$
	as $\theta(J)=-J$. But $\Ad J\not \in \Ad H$, as $J\not\in  Z(\SL_2)\C^*$, so the pair is not pleasant. Note that the condition of pleasantness can equivalently be stated as $\cA_\theta\subseteq H Z(G)$.
\end{example}

\begin{remark}
	It is possible to know, using techniques as in \cite{carmeli2015}, which elements inside $\cA_\theta$ are actually admissible. However, when looking at all possible pairs, any element in $\cA_\theta$ will be admissible for $\theta$ or for some other involution $\theta'$ of the same group. Thus, when dealing with all symmetric pairs, using $\cA_\theta$ instead of admissible elements makes no  difference.
\end{remark}

\subsection{Criteria for pleasantness}
\label{sec:criteria}
This section contains the criteria that will be used in Sections \ref{sec:classical-pleasant} and \ref{sec:exceptional-pleasant} to give a complete classification of pleasant pairs. We refer to these two sections for examples of their use.

The first one is the most obvious one.
\begin{lemma}\label{lemma:centerless}
	Any symmetric pair $(G,H, \theta)$ with $Z(G)=\{1\}$  is pleasant.
\end{lemma}

This lemma is actually a particular case of the following criterion. 
\begin{proposition}\label{prop:centerodd}
	Let $(G,H,\theta)$ be a symmetric pair. If $\theta(\lambda)=\lambda$ for $\lambda\in Z(G)$ (which is always the case for inner involutions) and $|Z(G)|$ is odd, then the pair is pleasant.
\end{proposition}

\begin{proof}
	Take an element $g\in \cA_\theta \setminus H$. We have that
	$$ \theta(g) = \lambda g $$
	for $\lambda\in Z(G)\setminus \{1\}$. By applying $\theta$ we get
	$$ g = \lambda^2 g,$$
	which is not possible as $|Z(G)|$ is odd. Hence $\cA_\theta= H$, and the pair is pleasant. 
\end{proof}

A similar argument gives a criterion when $\theta$ does not fix the centre element-wise.

\begin{proposition}\label{prop:centeroddeven}
	Let $(G,H,\theta)$ be a symmetric pair
	\begin{enumerate}
		\item[a)]  If  $|Z(G)|$ is odd and $\theta(\lambda)=\lambda^{-1}$ for any $\lambda\in Z(G)$, the pair is pleasant.
		\item[b)]  If there exists $g\in G$ such that $\theta(g)=\mu g$ for $\mu\in Z(G)\setminus \{1\}$ with $\theta(\mu)=\mu^{-1}$ and $\mu$ has no square root in $Z(G)$, then the pair is not pleasant.
	\end{enumerate}
\end{proposition}

\begin{proof}
	For part a), let  $g\in \cA_\theta \setminus H$ be such that $\theta(g) = \lambda g$ with $\lambda\in Z(G)\setminus \{1\}$. Consider an element $\sqrt{\lambda}\in Z(G)$ squaring to $\lambda$. The element $h=\sqrt{\lambda}g$  satisfies $$\theta(h)=\theta(\sqrt{\lambda}g)=(\sqrt{\lambda})^{-1}\lambda g=\sqrt{\lambda}g=h,$$
	that is, belongs to $H$, and $\Ad g =\Ad h $, as they differ by a central element.
	
	For part b), the element $g$ from the statement belongs to $\cA_\theta \setminus H$. If we had $\Ad g=\Ad h$ for $h\in H$, we would have $g=\nu h$ for $\nu\in Z(G)$, and also
	$$ \mu \nu h = \mu g = \theta(g) = \theta(\nu h) = \nu^{-1} h,$$
	This would mean $\nu^{2}=\mu^{-1}$, that is, $\mu^{-1}$, and also $\mu$,  must have a square root in $Z(G)$. Hence, if $\mu$ has not such a square root, the pair is not pleasant. 
\end{proof}

We give another criterion for inner involutions, motivated by an example.

\begin{example}\label{ex:SLpq}
	Consider the pair $(\SL_{p+q},\SSS(\GL_{p}+\GL_q))$.   We write elements as block matrices with the subdivision $(p,q)$ for rows and columns. The involution $\theta$ is given by $\Ad I_{p,q}$ for
	$$ I_{p,q}=\begin{pmatrix}[cc]
	I_p & 0     \\ 
	0 & -I_q
	\end{pmatrix}, \qquad  \theta : \begin{pmatrix}[cc]
	A & B    \\ 
	C& D
	\end{pmatrix} \mapsto \begin{pmatrix}[cc]
	A &  -B \\ 
	-C& D
	\end{pmatrix}. $$
	The matrices in $\cA_\theta$ should satisfy $\theta(g)=\lambda g$ for $\lambda$ a $(p+q)$-th root of unity.
	
	If $A\neq 0$ or $D\neq 0$ then $\lambda$ must be $1$ and $g\in H$. Otherwise, $\lambda=-1$ and $g$ would have the form $$\begin{pmatrix}[cc]
	0 & B     \\
	C & 0 
	\end{pmatrix}.$$
	When $p\neq q$, this matrix does not belong to $\SL_{p+q}$, as it is not invertible, so $\lambda=-1$ is not possible and, hence, the pair is  pleasant. 
\end{example}

\begin{proposition}\label{prop:innerauto}
	Let $(G,H,\theta)$ be a symmetric pair, such that $G\subseteq \GL(V)$ with $Z(G)\subseteq \C \cdot \Id$, and $\theta$ is an inner involution given by $\Ad k$ for some $k\in \GL(V)$ of order two. Let $V_+$ and $V_-$ be the $\pm 1$-eigenspaces of $k$ on $V$. 
	\begin{enumerate}
		\item[a)] 	If $\dim V_+\neq \dim V_-$, the pair is pleasant.
		\item[b)]   If $\dim V_+= \dim V_-$, $-\Id\in Z(G)$, and there is $g\in G$ interchanging $V_+$ and $V_-$, the pair is not pleasant.
	\end{enumerate}
\end{proposition}

\begin{proof}
	Consider $g\in \cA_\theta \setminus H$, there is $\lambda \in Z(G)\setminus \{1\}$ such that $\theta(g)=\lambda g$, that is,
	$$ kg = \lambda gk.$$
	By applying this identity to $v\in V_+$ and $w\in V_-$ we get
	$$ kgv = \lambda gv, \qquad kgw= -\lambda gw$$
	As the only eigenvalues of $k$ are $\pm 1$, we have that $\lambda$ must be $-\Id\in \GL(V)$ and $g$ must interchange $V_+$ and $V_-$. 
	
	For part a), if  $\dim V_+\neq \dim V_-$, the element $g$ would not be invertible as either 
	$$\dim g V_+ \leq \dim V_- < \dim V_+,\qquad  \textrm{ or } \qquad\dim gV_-\leq \dim V_+ < \dim V_-,$$
	so the pair is pleasant.
	
	For part b), when $\dim V_+= \dim V_-$, assume there is $g\in G$ interchanging $V_+$ and $V_-$. As $-\Id\in Z(G)$, it follows that $g\in \cA_\theta$. If we had $\Ad g=\Ad h$ for some $h\in H$, we would have $g=\nu h$  with $\nu\in Z(G)$, so $h$ would interchange $V_+$ and $V_-$. This is a contradiction, as $h\in H$ satisfies $kh=hk$, so it preserves, arguing as above, the subspaces $V_+$ and $V_-$.
\end{proof}

Clearly, if a symmetric pair is pleasant, any of its quotient symmetric pairs (in the sense of Definition \ref{def:quotient-sym-pair}) is pleasant. After recalling the classification of complex symmetric pairs, we will try to prove pleasantness for simply connected pairs, apart from the case of $\SO$ and $\Spin$, where it is preferable to deal first with $\SO$.

\subsection{Classification of complex symmetric pairs}
\label{sec:Satake}

Complex symmetric pairs for simple $G$ are classified by means of Satake diagrams. Take a maximal torus $T\subset G$ containing a maximal $\theta$-split torus $A$, that is, $\theta(x)=x^{-1}$ for $x\in A$. By \cite[\S 1]{vust}, assuming that $\theta$ is not the identity, we have that $A$ is not trivial and $T$ is $\theta$-stable. Consider the root system $\Delta:=\Delta(G,T)$.  Let 
$$\Delta_0:=\{ \al \in \Delta \st \al_{|A}=1 \} = \{ \al \in \Delta \st d\al_{|\fa}=0 \}$$
be the roots that are trivial when restricted to $A$ or, equivalently, $\fa$. Consider the action of $\theta$ on $\Delta$ denoted, in additive notation, by $$\alpha \mapsto -\theta\alpha,$$ which is given by $(-\theta\alpha)(t)=(\alpha(\theta(t)))^{-1}$ for $t\in T$.

\begin{definition}\label{def:theta-basis}
	A $\theta$-basis is a set of simple roots coming from an ordering such that
	$$ \text{for } \al\notin \Delta_0, \; \al>0 \Rightarrow  -\theta\alpha > 0.$$
\end{definition}

We choose one of such orderings (they always exist, as it can be shown by taking a lexicographic ordering) and consider the corresponding $\theta$-basis $\Pi$. By \cite[\S 1 and App.]{satake} or \cite[Prop. 32.9]{bump}, this choice implies that $\theta$ acts as an order $2$ permutation on $\Pi\setminus (\Pi\cap \Delta_0)$, and the Satake diagram is constructed as follows:
	\begin{itemize}
		\item consider the Dynkin diagram with respect to $\Pi$.
		\item colour black the nodes corresponding to $\al\in\Pi\cap \Delta_0$, whereas the rest of simple roots remain white.
		\item connect two simple roots permuted by $\theta$ with a grey bar (an alternative common notation is a double-headed arrow).  
	\end{itemize}

% whose restriction to $A$ is trivial, 

In Table \ref{tab:Satake} we show the Satake diagrams of complex symmetric pairs for $G$ simple, following \cite{araki}, together with the corresponding symmetric pair for $G$ simply connected\footnote{For $(E_8,D_8)$, the fixed subgroup is not $\SO(16)$ but a different quotient of $\Spin(16)$ by $\{\pm 1\}$}. Some degenerate indices may appear as in 	$(A_1,D_1)$, which refers to $(A_1,\C)$.

\begin{table}[htp]
	\begin{center}
				\begin{tabular}[280pt]{cc}
			\hline \hline		
			\parbox[t]{140pt}{\center \dynkin{A}{I}\\ \footnotesize  $(A_{2n},B_n)$, $(A_{2n-1},D_n)$ \\  $(\SL_{2n+1},\SO_{2n+1})$, $(\SL_{2n},\SO_{2n})$} & 	\parbox[t]{140pt}{\center \dynkin{A}{II}\\ \footnotesize $(A_{2n-1},C_n)$ \\ $(\SL_{2n},\
				\Sp_{2n})$}       \\ 
			& \\ \hline
			\parbox[t]{140pt}{ \center \dynkin{A}{IIIa}\\ \footnotesize $(A_{r+s+1},A_r+A_s+\C)$, $r\neq s$ \\ $(\SL_{r+s+2},\
				\SSS(GL_{r+1}\times \GL_{s+1}))$}  &   	\parbox[t]{140pt}{\center \dynkin{A}{IIIb}\\ \footnotesize $(A_{2r+1},A_r+A_r+\C)$ \\ $(\SL_{2r+2},\
				\SSS(GL_{r+1}\times \GL_{r+1}))$}     \\
			& \\ \hline
		\end{tabular}
		\begin{tabular}[280pt]{ccc}
			\parbox[t]{90pt}{\center \dynkin{C}{I}\\ \footnotesize $(C_{n},A_{n-1}+\C)$ \\ $(\Sp_{2n},\
				\GL_n)$} & \parbox[t]{90pt}{\center \dynkin{C}{IIa}\\ \footnotesize $(C_{r+s},C_r+C_s)$, $r\neq s$ \\ $(\Sp_{2(r+s)},\
				\Sp_{2r}\times \Sp_{2s})$} & \parbox[t]{90pt}{\center \dynkin{C}{IIb}\\ \footnotesize $(C_{2r},C_r+C_r)$ \\ $(\Sp_{4r},\
				\Sp_{2r}\times \Sp_{2r})$}  \\
			& \\ \hline
					\end{tabular}
						\begin{tabular}[280pt]{ccc}
					\parbox[t]{138pt}{ \center \footnotesize $(B_{r+s},B_{r}+D_s)$ \\ $(\Spin_{2(r+s)+1},				\Spin_{2r+1}\times_{\Z_2} \Spin_{2s})$} & \parbox[t]{70pt}{\center \dynkin{B}{I}\\ \footnotesize $r\neq s-1,s$}     & \parbox[t]{62pt}{\center \dynkin{B}{II}\\ \footnotesize $r=s-1,s$}  \\
					& \\ \hline
				\end{tabular}
					\begin{tabular}[280pt]{cccc}
				\parbox[t]{142pt}{\center \footnotesize $(D_{r+s},D_{r}+D_s)$ \\ $(\Spin_{2(r+s)},\Spin_{2r}\times_{\Z_2} \Spin_{2s}),$\\ $\cdot$ \\ $(D_{r+s+1},B_{r}+B_s)$\\ $(\Spin_{2(r+s+1)},\Spin_{2r+1}\times_{\Z_2} \Spin_{2s+1})$} & \parbox[t]{62pt}{\center \dynkin{D}{Ia}\\ \bigskip \footnotesize \vspace{1pt} $r\neq s-1,s$ } & \parbox[t]{30pt}{\center \dynkin{D}{Ib}\\ \bigskip \footnotesize \vspace{2pt} $r=s-1$} & \parbox[t]{28pt}{\center \dynkin{D}{Ic}\\ \bigskip \footnotesize \vspace{3pt} $r=s$}    \\
				& \\ \hline
			\end{tabular}
							\begin{tabular}[280pt]{ccc}
			\parbox[t]{120pt}{\center \footnotesize \bigskip $(D_{r},A_{r-1}+\C)$, \\ $(\Spin_{2r}, \GL_{r})$} & \parbox[t]{75pt}{\center \dynkin{D}{IIIa}\\ \footnotesize   $r$ even } & \parbox[t]{75pt}{\center \dynkin{D}{IIIb}\\ \footnotesize $r$ odd}     \\
			& \\ \hline \hline
		\end{tabular}
	\begin{tabular}[280pt]{ccc}
		\parbox[t]{90pt}{\center \dynkin{G}{I}\\ \footnotesize $(G_2,A_1+A_1)$ \\ $(G_2,\SL_1\times_{\Z_2} \SL_1)$} & \parbox[t]{90pt}{\center \dynkin{F}{I}\\ \footnotesize $(F_4,B_4)$ \\ $(F_4,\Spin_9)$} & \parbox[t]{90pt}{\center \dynkin{F}{II}\\ \footnotesize $(F_4,C_3+A_1)$ \\ $(F_4,\Sp_6\times_{\Z_2} \SL_2)$}    \\
		& \\ \hline
	\end{tabular}
\begin{tabular}[280pt]{cccc}
	\parbox[t]{50pt}{\center \dynkin{E}{I}\\ \footnotesize $(E_{6},C_4)$ \\ $(E_6,\Sp_8/\Z_2)$} & \parbox[t]{80pt}{\center \dynkin{E}{II}\\ \footnotesize $(E_6,A_5+A_1)$ \\ $(E_6,\SL_6\times_{\Z_2} \SL_2)$} & \parbox[t]{50pt}{\center \dynkin{E}{IV}\\ \footnotesize $(E_6,F_4)$ \\ $(E_6,F_4)$} & \parbox[t]{80pt}{\center \dynkin{E}{III}\\ \footnotesize $(E_6,D_5+\C)$ \\ $(E_6,\Spin_{10}\times_{\Z_4} \GL_1)$}    \\
	& \\ \hline
\end{tabular}
			\begin{tabular}[280pt]{ccc}
		\parbox[t]{90pt}{\center \dynkin{E}{V}\\ \footnotesize $(E_{7},A_7)$ \\ $(E_7,\SL_8/\Z_2)$} & \parbox[t]{90pt}{\center \dynkin{E}{VI}\\ \footnotesize $(E_7,D_6+A_1)$ \\ $(E_7,\Spin_{12}\times_{\Z_2} \SL_2)$} & \parbox[t]{90pt}{\center \dynkin{E}{VII}\\ \footnotesize $(E_7,E_6+\C)$ \\ $(E_7,E_6\times_{\Z_3} \GL_1)$}   \\
		& \\ \hline
	\end{tabular}
	\begin{tabular}[280pt]{cc}
	\parbox[t]{140pt}{\center \dynkin{E}{VIII}\\ \footnotesize $(E_8,D_8)$\\ $(E_8,\Spin_{16}/\{\pm 1\})$} & 	\parbox[t]{140pt}{\center \dynkin{E}{IX}\\ \footnotesize $(E_8,E_7+A_1)$\\ $(E_8,E_7\times_{\Z_2} \SL_2)$}       \\ 
	& \\ \hline \hline 
	\end{tabular}
	\end{center}
\vspace{1pt}
\caption{Satake diagrams}
\label{tab:Satake}
\end{table}

\begin{remark}\label{rem:Satake-extreme}
	There are two cases that help understand the construction of Satake diagrams and will be used later. For $G$ a simple group, the Satake diagram of the trivial pair $(G,G)$ is the Dynkin diagram of $G$ with all the nodes coloured in black, whereas the Satake diagram of $(G\times G,\Delta G)$ consists of two copies of the Dynkin diagram of $G$, with all nodes white, with the corresponding nodes connected by a bar. For example, for $G=\SL_n$ or of type $A_n$ we would have the following diagrams.
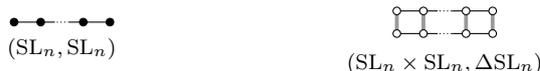
\begin{center}
\parbox[t]{140pt}{\center \medskip \dynkin{A}{}\\ \footnotesize    $(\SL_{n},\SL_{n})$}   \parbox[t]{140pt}{\center \begin{tikzpicture}
	\dynkin[name=upper]{A}{I}
	\node (current) at ($(upper root 1)+(0,-.3cm)$) {};
	\dynkin[at=(current),name=lower]{A}{I}
	\begin{scope}[on background layer]
	\foreach \i in {1,...,4}%
	{%
		\draw[/Dynkin diagram/foldStyle]
		($(upper root \i)$) -- ($(lower
		root \i)$);%
	}%
	\end{scope}
	\end{tikzpicture}\\ \footnotesize  $(\SL_{n}\times \SL_n,\Delta\SL_{n})$}
\end{center}
\end{remark}

Conversely, the Satake diagram, together with the Lie algebras $\fa$ and $\ft$, determines the complex symmetric pair at the level of Lie algebras. For $G$ simply connected, it is possible to recover the involution, and hence $H$, from the subgroups $A$ and $T$, together with the Satake diagram. However, it is not possible to easily read $H$ from the Satake diagram. 

When $G$ is not simply connected, one has to be especially careful, as Lie algebra pairs related by an outer automorphism could yield symmetric pairs that are not isomorphic and may have indeed  very different properties. This is the case of
$$\dynkin{D}{oo**}$$
when it gives the pairs $(\SO_8,\GL_4)$ and $(\SO_8,\SSS(\OO_6\times \OO_2))$, whose different behaviour will be exhibited in Section \ref{sec:classical-pleasant}. However, since our main focus is on simply connected pairs, we leave further study of this issue for future work.

%***************  FINALLY NOT NEEDED ******************
%The involution $\theta$ determines the Weyl group $W_\theta$, consisting of 
%$$W_\sigma = \{ s\in W(\Delta)$          sA = A  \}$$
%Possible choices of $\theta$-basis correspond to the action of $W_\theta$. 

\begin{remark}
	Recall that complex symmetric pairs $(G,H,\theta)$ are the complexification of Riemannian symmetric pairs, where $(G_r,H_r,\theta_r)$ where $G_r$ is a real form of $G$, the group $H_r$ is a maximal compact subgroup of $G_r$ and $\theta_r$ is the corresponding involution.
\end{remark}

%\begin{remark}
%	Probably $(C_1,\C)=(A_1,\C)$ but not $(B_1,B_0+D_1)$, inner versus outer.
%	
%	$(A_3,D_2)$ outer, , $(D_3,B_1+B_1)$ the same probably
%	
%	 $(A_3,A_1+A_1+\C)$ inner, and still ... quite different
%	
%	$(D_3,A_2+\C)$ and $(A_3,A_2+\C)$
%	
%	$(D_4,A_3+\C)$ and $(D_4,D_3+\C)$	
%	
%\end{remark}

From now on, we will mostly use Lie types to refer to symmetric pairs. For instance, $(\SL_2,\C^*)$ is denoted by $(A_1,\C)$ when it is clear that the group is $\SL_2$. We will mostly, apart from the orthogonal and Spin groups, refer to $G$ simply connected when dealing with pairs.

\subsection{Linear, symplectic and orthogonal pleasant pairs}
\label{sec:classical-pleasant}

 For the sake of simplicity, we first study symmetric pairs $(G,H)$ where $G$ is $\SL$, $\Sp$ or $\SO$. We shall use their Lie type  in the proofs (the correspondence can be read from Table \ref{tab:Satake}).

We fix some notation. Let $I_p$ denote the identity matrix of rank $p$, we define
$$ I_{p,q}=\begin{pmatrix}[cc]
I_p & 0     \\ 
0 & -I_q
\end{pmatrix}, \qquad J=\begin{pmatrix}[cc]
0 & -I_p     \\ 
I_p &  0
\end{pmatrix}.$$

% P_{n-1,n}=\begin{psmallmatrix}
%1 &  & & &    \\ 
%& \ddots & & &    \\ 
%&  & 1 & &    \\ 
%&  &  & & 1    \\ 
%&  & & 1 &     
%\end{psmallmatrix}.

\begin{proposition}\label{prop:pleasant-linear}
	The symmetric pairs $(\SL_{2n+1}, \SO_{2n})$ and $(\SL_{2n},\Sp_{2n})$ are pleasant, whereas the pair $(\SL_{2n}, \SO_{2n})$ is not pleasant. The pair $(\SL_{p+q},\SSS(\GL_{p}+\GL_{q}))$ is pleasant if and only if $p\neq q$.
\end{proposition}

\begin{proof}
	We look at all the cases:

				$\bullet$ $(A_{2n},B_{n})$ for $n\geq 1$. We have $Z(\SL_{2n+1})=\{ \lambda \Id \st \lambda^{2n+1}=1 \}$, so $|Z(\SL_{2n+1})|$ is odd, and $\theta(g)=(g^\intercal)^{-1}$, which sends $\lambda \Id\in Z(\SL_{2n+1})$ to $\lambda^{-1}\Id$. By Proposition \ref{prop:centerodd}a), the pair is pleasant.	
	
	$\bullet$  $(A_{2n-1},D_{n})$ for $n\geq 1$. In this case we have that  $|Z(\SL_{2n})|$ is even and $\theta(g)=(g^\intercal)^{-1}$. Take a primitive $4n$-th root of unity $\alpha$ and consider, with a block partition $(2n-1,1)$, the element  $$g= \begin{pmatrix}[cc]
	\al^{-1} &  0    \\ 
	0   & -\al^{-1}
	\end{pmatrix}\in \SL_{2n}.$$ We have $\theta(g)= (\al^2 \Id) g$, so $g\in \cA_\theta$, as $\al^{2} \Id\in Z(G)$. Since $\al^{2} \Id$ has no square root in $Z(G)$, by Proposition \ref{prop:centeroddeven}b), the pair is not pleasant.
	
$\bullet$ $(A_{2n-1},C_n)$. The involution in $\SL_{2n}$ is given by $\theta(g)=(\Ad J)((g^\intercal)^{-1})$. Consider $g\in \SL_{2n}$ such that $\theta(g)=gz$, with $z\in Z(\SL_{2n})$. We can take $\mu\in \GL_{2n}$ with $\mu^2=z$, so that $g\mu\in \GL_{2n}$ satisfies $\theta(g\mu)=g\mu$ for the same involution $\theta$ in $\GL_{2n}$. But since the fixed points of $\theta$ in $\GL_{2n}$ are also $\Sp_{2n}$, we have that $g\mu\in \Sp_{2n}$ and $\Ad g=\Ad g\mu \in \Ad H$.
	
	$\bullet$ $(A_{r+s+1},A_{r}+A_s+\C)$.  This is actually the pair of Example \ref{ex:SLpq}. By Proposition \ref{prop:innerauto}, this pair is pleasant if and only if $r\neq s$. For $r=s$, consider, with $n=r+1$, the $(n,n)$-block matrix 
	$g= \begin{pmatrix}[cc]
	0 &  (-1)^{n}   \\ 
	1  & 0
	\end{pmatrix}$
	\end{proof}

\begin{remark}
For each subgroup $F\subseteq Z(\SL_n)$, there are symmetric pairs associated to $\SL_n/F$ of any of the types discussed above. We check whether some of them may become pleasant when passing to the quotient.
For $(A_{2n-1},D_{n})$ and $(A_{2n-1},C_n)$, the centre of $\SL_{2n}/F$ is cyclic, the argument from the proof of Proposition \ref{prop:pleasant-linear} still applies if  $|Z(\SL_{2n}/F)|$ is even, whereas Proposition \ref{prop:centerodd}a) applies if $|Z(\SL_{2n}/F)|$ is odd. For $(A_{2r+1},A_{r}+A_r+\C)$, the same applies as long as $-\Id\in \SL_{2r+2}$, so if $|Z(\SL_{2r}/F)|$ is even, the pair is not pleasant, and if $|Z(\SL_{2r}/F)|$ is odd, the pair is  pleasant
\end{remark}

\begin{proposition}\label{prop:pleasant-symplectic}
	The symmetric pair  $(\Sp_{2(r+s)},\Sp_{2r}\times \Sp_{2s})$ is pleasant if and only if $r\neq s$, whereas $(\Sp_{2n}, \GL_{n})$ is not pleasant.
\end{proposition}

\begin{proof}
	As before, we look at the possible cases:

		$\bullet$ $(C_{r+s},C_r+C_s)$. The involution is $\theta=\Ad {\begin{psmallmatrix}
		I_{r,s} &  0    \\ 
			0  & I_{r,s}
			\end{psmallmatrix}}.$ By Proposition \ref{prop:innerauto}, this pair is pleasant if and only if $r\neq s$. When $r=s$, one can consider $$g= \begin{pmatrix}[cccc]
		0 &  1   & 0 &  0 \\ 
		1 &  0   &  0&  0 \\ 
		0 &  0   & 0& 1 \\ 
		0 &  0   & 1 &0 
\end{pmatrix}.$$

		$\bullet$ $(C_n,A_{n-1}+\C)$.  The involution is $\theta=\Ad {I_{n,n}}$. By Proposition \ref{prop:innerauto}b) with $g= \begin{pmatrix}[cc]
		0 &  1    \\ 
    	-1  & 0
		\end{pmatrix}$, the pair is not pleasant.	
				
\end{proof}

\begin{proposition}\label{prop:pleasant-orthogonal}
	The symmetric pairs $(\SO_{p+q},\SSS(\OO_p\times \OO_q))$ are pleasant if and only if $p\neq q$, whereas $(\SO_{2r},\GL_r)$ for $r\geq 4$ is not pleasant.
\end{proposition}

\begin{proof}
Just as before, 

$\bullet$ $(B_{r+s},B_r+D_s)$, $(D_{r+s},D_{r}+D_s)$, $(D_{r+s+1},B_{r}+B_s)$, that is, the pairs of the form $(\SO_{p+q},\SSS(\OO_p\times \OO_q))$. By Proposition \ref{prop:innerauto}, they are pleasant if and only if $p\neq q$. When $p=q$, consider the element the determinant-one element $g= \begin{pmatrix}[cc]
0 &  i^{2p}    \\ 
i^{2p}  & 0
\end{pmatrix}$.

$\bullet$ $(D_{r},A_{r-1}+\C)$.  By Proposition \ref{prop:innerauto}b) with $g= \begin{pmatrix}[cc]
0 &  i^{2p}    \\ 
i^{2p}  & 0
\end{pmatrix}$, the pair is not pleasant.

\end{proof}

\subsection{Spin pleasant pairs}
\label{sec:spin-pairs}

Although the Spin group can be seen as a group of matrices via the spinor representation, it will be simpler to prove or disprove pleasantness using Proposition \ref{prop:pleasant-orthogonal} about $\SO$ and the covering map $\Spin\to \SO$. 

As usual, we realize the Spin group inside the Clifford algebra of a non-degenerate quadratic vector space $(V,Q)$,
 $$\Cl(V,Q)= \frac{\otimes^\bullet V}{gen(v\otimes v - Q(v)\cdot 1)}.$$
The Spin group corresponds to 
$$ \Spin(V,Q)= \{ v_1\ldots v_{2l} \st Q(v_i)=\pm 1 \}, $$
where juxtaposition denotes the Clifford product.

Since we are working over an algebraically closed field, all quadratic forms over the same vector space are equivalent, so we talk simply about $\Spin_n$.

The centre of the Spin group depends on the parity of $n$. For $n=2m+1$ odd, 
\begin{equation}\label{eq:center-Spin2m+1}
Z(\Spin_{2m+1}) = \{ \pm 1\} \cong \Z_2,
\end{equation}
For $n=2m$ even, we have
$$Z(\Spin_{2m}) = \{ \pm 1, \pm \omega \},$$
where $\omega$ is the product of a chosen orthonormal basis $\{e_1,\ldots,e_{2m}\}$ of $V$,
$$ \omega= e_1 \ldots e_{2m} \in \Spin_{2m}, $$
which is well defined up to sign.

As $e_ie_j+e_je_i=\delta^i_j$, we have that $\omega^2 = 1$ for $m$ even and $\omega^2=-1$ for $m$ odd, so 
\begin{align}\label{eq:center-Spin4m-4m+2}
Z(\Spin_{4m})&\cong \Z_2\times \Z_2,& Z(\Spin_{4m+2})&\cong \Z_4.
\end{align}
We check that the involution $\theta$ on $\Spin_{r+s}$ corresponding to the Lie algebra pair 
\begin{equation}\label{eq:fsor+s}
(\fso_{r+s}, \fso_r\oplus \fso_s)
\end{equation}
is given by the action of $I_{r,s}$ by
$$ v_1\ldots v_{2l} \mapsto (I_{r,s} v_1)\ldots (I_{r,s} v_{2l}).$$
Let $\pi:\Spin_n\to \SO_n$ be the $2:1$-covering map given by $$\pi: g\mapsto  (x\mapsto gxg^{-1}).$$
We have $\pi(v_1\ldots v_{2l}) = R_{v_1} \circ \ldots\circ R_{v_{2l}},$ where $R_v$ denotes the reflection with respect to $v$. Since 
$$ I_{r,s}\circ R_ v \circ I_{r,s} = R_{I_{r,s}v},$$
we have, with $\Ad {I_{r,s}}$ denoting the conjugation by $I_{r,s}$ both on $\SO$ or $\fso$, that 
\begin{center}
	\begin{tikzcd}\Spin_{r+s} \arrow[rr, "\theta"]\arrow[d, "\pi"]& & \Spin_{r+s} \arrow[d, "\pi"] \\ \SO_{r+s} \arrow[rr,"\Ad {I_{r,s}}"]& & \SO_{r+s} \\
		\fso_{r+s}\arrow[u,"exp"] \arrow[rr,"\Ad {I_{r,s}}"]& & \fso_{r+s} \arrow[u,"exp"] \end{tikzcd}
\end{center}
% $$ \pi \circ \theta = (\Ad {I_{r,s}})\circ \pi,$$
commutes, that is, $\theta$ covers the involution on $\SO_{r+s}$ and hence $\fso_{r+s}$.

The $\Spin$ pair corresponding to the Lie algebra pair \eqref{eq:fsor+s} is 
$$ (\Spin_{r+s},\Spin_r \times_{\Z_2} \Spin_s).$$ The quotient by $\Z_2$ can be easily understood from the map 
\begin{align*}
\Spin_r \times \Spin_s & \to \Spin_{r+s} \\
 (v_1\ldots v_{2l} , w_1\ldots w_{2m}) &\mapsto v_1\ldots v_{2l}w_1\ldots w_{2m},  
\end{align*}
whose kernel is $\la (-1,-1)\ra \cong \Z_2$. This pair can be compared to the orthogonal pair $(\SO_{r+s},\SSS(\OO_r\times \OO_s))$ along the lines of Example \ref{ex:SO-PSO}.

The next step in order to apply the criteria of Section \ref{sec:criteria} is to know how the involution acts on the centre. Recall that the element $-1$ is represented in  $\Spin$, for some $v$ such that $Q(v)=\pm 1$, by
$$ -1  = (-v)\cdot (v/Q(v)).$$
Hence, we have
$$\theta(-1)=\theta((-v)\cdot (v/Q(v)))= (-I_{r,s}v) \cdot (I_{r,s}v/Q(v)) = -1,$$
as $I_{r,s}$ is an isometry.
When $n=r+s$ is even, by choosing $\{e_1,\ldots,e_{r+s}\}$ the same basis in which $I_{r,s}$ is expressed,  we have 
\begin{equation}\label{eq:theta-omega}
 \theta(\omega) = \theta(e_1\ldots e_{r+s}) = (-1)^s e_1\ldots e_{r+s} = (-1)^s \omega.
\end{equation}

\begin{proposition}
	A symmetric pair associated with $\Spin_{r+s}$ is pleasant if and only if it is $(\Spin_{r+s},\Spin_r \times_{\Z_2} \Spin_s)$ with $r+s=4q+2$ and $r\neq s$ odd numbers. 
\end{proposition}

\begin{proof}
	To start with, the pair of type $(D_r,A_{r-1}+\C)$ is not pleasant for $\Spin$, as neither is for $\SO$. We shall only look at the pairs $(\Spin_{r+s},\Spin_r \times_{\Z_2} \Spin_s)$.
	
	For $n=r+s$, with $r,s>0$, take $v,w$ such that $Q(v)=\pm 1$, $I_{r,s} v= v$, $Q(w)=\pm 1$ and $I_{r,s} w = -w$. Consider $g=vw\in \Spin_{n}$. We have
	\begin{equation}\label{eq:theta(g)}
	\theta(g)=\theta(vw)=v(-w)=-g.  
	\end{equation}
	By \eqref{eq:center-Spin2m+1}, \eqref{eq:center-Spin4m-4m+2} and Proposition \ref{prop:centeroddeven}b), the pair is not pleasant when $n$ is odd or $n=4q$.
	
	For $n=r+s$ even of the form $4q+2$, we have $\omega^2=-1$. Consider $g$ as in \eqref{eq:theta(g)}, if we had $g=\nu h$ for $\nu\in Z(\Spin_n)=\la \omega\ra\cong \Z_4$, we would have 
	$$-\nu h= -g = \theta(g)=\theta(\nu)h,$$ 
	that is, $\theta(\nu)=-\nu$.
For $s$ even, this is not possible, as $\theta(\omega)=\omega$ by \eqref{eq:theta-omega}. Whereas for $s$ odd, we have $\theta(\omega)=-\omega$, so $\omega g$ satisfies $\theta(\omega g)=\omega g$ and $\Ad g=\Ad \omega g$.

We continue with $n=r+s=4q+2$ with $s$ odd, which is the only possible pleasant case. An element $g\in \cA_\theta \setminus H$ could also possibly satisfy
\begin{equation}\label{eq:theta-g-omega}
 \theta(g)=\pm\omega g.
\end{equation}
 By applying the projection $\pi$ to $\SO_{r+s}$ we get $\pi(\theta(g))=\pi(\omega g)$, that is,
$$ \theta'(\pi(g)) = - \pi(g),$$ 
where $\theta'$ denotes the involution for $\SO_{r+s}$. Recall, from the proof of Proposition \ref{prop:pleasant-orthogonal}, that this is only possible when $r=s$, and in this case we already know that the pair associated to $\Spin$ is not pleasant, as the pair associated to $\SO$ is not pleasant. For $r\neq s$ odd numbers, and $r+s=4q+2$, we have that  \eqref{eq:theta-g-omega} is not possible, so $(\Spin_{r+s},\Spin_r\times_{\Z_2} \Spin_s)$ is pleasant.
\end{proof}

\subsection{Exceptional Lie groups}
\label{sec:exceptional-lie-groups}

In this section we present the exceptional Lie groups  following \cite{yokota}, and focusing especially on $E_6$ and $E_7$.

Let $\Oc$ be the division algebra of octonions, which we see as the Cayley-Dickson process applied to the quaternions $\HH$, that is, $\Oc=\HH\oplus \HH l$ with $l^2=-1$. Consider the conjugation $x\mapsto \overline{x}$ of $\Oc$ fixing only the real numbers. Let $\Oc_\C$ denote the complexification of the octonions. The automorphism group of $\Oc_\C$ is
$$G_2:=\Aut(\Oc_\C)=\{ \al\in \Iso_\C(\Oc_\C) \st \alpha(xy)=\alpha(x)\alpha(y) \}.$$

Define a conjugation on $\Oc_\C$ as the $\C$-linear extension of the conjugation on $\Oc$,
$$ x+iy\mapsto \overline{x+iy}=\overline{x}+i\overline{y}.$$ 
Let $\fJ$ be the $27$-dimensional complex subspace of $3\times 3$ matrices over $\Oc_\C$ given by 
 $$\fJ=\left\{ \begin{pmatrix}
 \xi_1 & x_3 & \overline{x}_2 \\
 \overline{x}_3 & \xi_2& x_1 \\
 x_2 & \overline{x}_1  & \xi_3 \\
 \end{pmatrix} \st \xi_i \in \C, x_i\in \Oc_\C \right\},$$
 with the commutative group operation 
 $$ X\circ Y = \frac{1}{2} (XY + YX).$$ 
 This is actually the complexification of the hermitian $3\times 3$ matrices over $\Oc$, and is usually referred to as the complex exceptional Jordan algebra. The group of automorphisms of $\fJ$ is
 $$F_4:= \Aut(\fJ) = \{ \al\in Iso_\C(\fJ) \st \alpha(X\circ Y)=\alpha(X)\circ \alpha(Y) \}.$$
 
The algebra $\fJ$ has a determinant operator
 $$ det: \fJ\to \Oc_\C,$$
 whose polarization is denoted by $(\cdot,\cdot,\cdot):\fJ^3\to \C$. The group of determinant-preserving automorphisms is
$$ E_6 := \{ \alpha\in \Iso_\C(\fJ) \st \det(\alpha X)=\det X \textrm{ for all } X\in \fJ \}.$$
It satisfies $Z(E_6)=\la \omega \Id \ra \cong \Z_3$ where $\omega$ is a third root of unity. Its Lie algebra is
$$ \fe_6 = \{ \al\in \End_\C(\fJ) \st (\al X,X,X)=0 \}.$$

Consider the complex Freudenthal vector space
$$ \fB = \fJ \oplus \fJ \oplus \C \oplus \C,$$
whose automorphisms we denote by complex  $(27,27,1,1)$-block matrices,
\begin{equation}\label{eq:matrixE7}
\begin{pmatrix}[cccccc]
\multicolumn{2}{c}{\multirow{2}{*}{$A$}} & \multicolumn{2}{c}{\multirow{2}{*}{$B$}} & \multirow{2}{*}{$k$} & \multirow{2}{*}{$l$}  \\
& & & & \\ 
\multicolumn{2}{c}{\multirow{2}{*}{$C$}} & \multicolumn{2}{c}{\multirow{2}{*}{$D$}} & \multirow{2}{*}{$p$} & \multirow{2}{*}{$q$}  \\
& & & &\\ 
\multicolumn{2}{c}{r}  & \multicolumn{2}{c}{s} & a & b\\ 
\multicolumn{2}{c}{u} & \multicolumn{2}{c}{v} & c & d
\end{pmatrix}.
\end{equation}
In order to define $E_7$, we endow $\fJ$ with some extra structure. For $X,Y\in \fJ$, let $(X,Y)=tr(X\circ Y)$ be a bilinear form on $\fJ$, and  define the Freudenthal product by
$$ X\times Y :=\frac{1}{2}(2X\circ Y -tr(X)Y-tr(Y)X + (tr(X)tr(Y)-(X,Y)\Id)).$$
Any $X\in \fJ$ defines $\widetilde{X}\in\Hom_\C(\fJ)$ by left multiplication, $\widetilde{X}(Y)=X\circ Y$, for $Y\in\fJ$. The element defined by
$$ X \vee Y := [\widetilde{X},\widetilde{Y}] + (X\circ Y - \frac{1}{3}(X,Y)\Id)\widetilde{\phantom{a}} $$
belongs to $\fe_6$, so we get an operation $\vee:\fJ\times \fJ\to \fe_6$.

For two elements $P=(X,Y,\xi,\eta)$, $Q=(Z,W,\zeta,\omega)$ in $\fB$, define a linear mapping $P\times Q:\fB\to \fB$ by 
\begin{equation}\label{eq:Phi-E7}
P\times Q:=\Phi(\phi,A,B,\nu) := \begin{pmatrix}[cccccc]
\multicolumn{2}{c}{\multirow{2}{*}{$\phi-\frac{1}{3}\nu$}} & \multicolumn{2}{c}{\multirow{2}{*}{$2B\times $}} & \multirow{2}{*}{$0$} & \multirow{2}{*}{$A$}  \\
& & & & \\ %\hline
\multicolumn{2}{c}{\multirow{2}{*}{$2A \times $}} & \multicolumn{2}{c}{\multirow{2}{*}{$-^t\phi+\frac{1}{3}\nu$}} & \multirow{2}{*}{$B$} & \multirow{2}{*}{$0$}  \\
& & & &\\ %\hline
\multicolumn{2}{c}{0}  & \multicolumn{2}{c}{(A,\cdot)} & \nu & 0\\ %\hline
\multicolumn{2}{c}{(B,\cdot )} & \multicolumn{2}{c}{0} & 0 & -\nu
\end{pmatrix},
\end{equation}
where $\phi=-\frac{1}{2}(X\vee W + Z\vee Y)$, $A=-\frac{1}{4}(2Y\times W-\xi Z -\zeta X)$, $B=\frac{1}{4}(2X\times Z - \eta W -\omega Y)$ and $\nu=\frac{1}{8}((X,W)+(Z,Y)-3(\xi\omega+\zeta\eta))$.

%\begin{pmatrix}
%	X\\Y\\
%	\xi\\ \eta
%\end{pmatrix}

\begin{definition}
	The exceptional Lie group $E_7$ is given by
	$$ E_7 :=\{ \al \in \mathrm{Iso}_\C(\fB) \st \al(P\times Q)\al^{-1} = \al P \times \al Q \textrm{ for all } P,Q \in \fB\}.$$
\end{definition}
{\noindent It satisfies $Z(E_7)=\{\pm \Id\}\cong Z_2$, and its Lie algebra is denoted by $\fe_7$.}

The group $E_8$ is defined in terms of its Lie algebra, whose underlying $248$-dimensional vector space is $$\fe_8=\fe_7\oplus \fB\oplus \fB\oplus \C\oplus \C\oplus \C,$$ and whose bracket is given as in \cite[Sec. 5]{yokota}. We then have 
$$ E_8 := \{ \al \in\Iso_\C(\fe_8) \st \al([X,Y])=[\al(X),\al(Y)] \textrm{ for all } X,Y\in \fe_8 \}. $$
A very important property is that the groups $G_2$, $F_4$ and $E_8$ have trivial centre.

The definitions above will allow us to describe, when necessary, the involutions of the corresponding exceptional symmetric pairs and determine their pleasantness.

\subsection{Exceptional pleasant pairs}
\label{sec:exceptional-pleasant}

The following theorem sums up the study of pleasantness for exceptional complex symmetric pairs.

\begin{theorem}\label{theo:exceptional-pleasant}
	All the exceptional symmetric pairs $(G,H)$ for simply-connected $G$ are pleasant apart from $(E_7,E_6+\C)$ and $(E_7,A_7)$.
\end{theorem}

\begin{proof}
	We first apply Lemma \ref{lemma:centerless} for the centreless groups $G_2$, $F_4$ and $E_8$.
	%We get  that $(G_2,A_1+A_1)$, $(F_4,B_4)$, $(F_4,C_3+A_1)$, $(E_8,D_8)$ and $(E_8, E_7+A_1)$ are pleasant.
	Secondly, recall that $E_6$ has centre $\Z_3$. Thus, for any symmetric pair $(E_6,H)$ the involution $\theta$ can either fix or invert the centre. By Propositions \ref{prop:centerodd} and \ref{prop:centeroddeven}a), any pair with $G=E_6$ is pleasant. 
	
	Finally, for $E_7$ we have that the centre is $\Z_2$ and we have to look case by case. Just as in Section \ref{sec:exceptional-lie-groups}, a good reference for the involutions is \cite{yokota}.
	
	$\bullet$ $(E_7,D_6+A_1)$. Define first the linear involution $\sigma:\fJ\to\fJ$ 
	$$\sigma: \begin{pmatrix}
		\xi_1 & x_3 & \bar{x}_2 \\
		\bar{x}_3 & \xi_2& x_1 \\
		x_2 & \bar{x}_1  & \xi_3 \\
	\end{pmatrix} \mapsto \begin{pmatrix}
	\xi_1 & -x_3 & -\bar{x}_2 \\
	-\bar{x}_3 & \xi_2& x_1 \\
	-x_2 & \bar{x}_1  & \xi_3 \\
\end{pmatrix}.$$
	
		Denote the $\pm 1$-eigenspaces of $\sigma$ on $\cJ$ by $R$ and $S$, respectively, which are $11$ and $16$-dimensional.  	Extend this involution to $\sigma:\fB\to\fB$ by
	$$\sigma(X,Y,\xi,\eta)=(\sigma X,\sigma Y, \xi, \eta),$$ and consider the involution $\theta:E_7\to E_7$ given by
	$$		\theta(\alpha)=\sigma\alpha\sigma. $$
		By using the block-matrix partition \eqref{eq:matrixE7}, the involution $\theta$ acts by 
	$$\theta: \begin{pmatrix}[cccccc]
\multicolumn{2}{c}{\multirow{2}{*}{$A$}} & \multicolumn{2}{c}{\multirow{2}{*}{$B$}} & \multirow{2}{*}{$k$} & \multirow{2}{*}{$l$}  \\
& & & & \\ 
\multicolumn{2}{c}{\multirow{2}{*}{$C$}} & \multicolumn{2}{c}{\multirow{2}{*}{$D$}} & \multirow{2}{*}{$p$} & \multirow{2}{*}{$q$}  \\
& & & &\\ 
\multicolumn{2}{c}{r}  & \multicolumn{2}{c}{s} & a & b\\ 
\multicolumn{2}{c}{u} & \multicolumn{2}{c}{v} & c & d
\end{pmatrix} \mapsto  	\begin{pmatrix}[cccccc]
	\multicolumn{2}{c}{\multirow{2}{*}{$\sigma A\sigma $}} & \multicolumn{2}{c}{\multirow{2}{*}{$\sigma B\sigma $}} & \multirow{2}{*}{$\sigma k$} & \multirow{2}{*}{$\sigma  l$}  \\
	& & & & \\
	\multicolumn{2}{c}{\multirow{2}{*}{$\sigma C\sigma $}} & \multicolumn{2}{c}{\multirow{2}{*}{$\sigma D\sigma $}} & \multirow{2}{*}{$\sigma p$} & \multirow{2}{*}{$\sigma q$}  \\
	& & & &\\ 
	\multicolumn{2}{c}{r\sigma }  & \multicolumn{2}{c}{s\sigma } & a & b\\ 
	\multicolumn{2}{c}{u\sigma } & \multicolumn{2}{c}{v\sigma } & c & d
\end{pmatrix}.$$

		We need to look at $g\in E_7$ such that $\theta(g)=-g$. 		For $\theta(g)=-g$ to be possible we need that:
		\begin{itemize}
		 \item $A,B,C,D$ interchange $R$ and $S$,
		 \item $k,l,p,q$ are given by the inner product by an element of $S$,
		 \item $r,s,u,v$ annihilate $R$,
		 \item $a,b,c,d$ vanish.
		\end{itemize}
We can then check that $$g(S\oplus S\oplus \{0\} \oplus \{0\})\subseteq R\oplus R \oplus \C \oplus \C,$$ but this means that $g$ cannot be invertible as
		$$ \dim (S\oplus S\oplus \{0\} \oplus \{0\}) > \dim (R\oplus R \oplus \C \oplus \C).$$ Hence, there are no $g\in E_7$ such that $\theta(g)=-g$ and the pair is pleasant.

		$\bullet$ $(E_7,E_6+\C)$. Define a linear involution $\iota:\fB\to\fB$ by
		$$\iota(X,Y,\xi,\eta)=(-iX,iY,-i\xi,i\eta),$$ and consider the involution $\theta:E_7\to E_7$ given by
	$$		\theta(\alpha)=\iota\alpha\iota^{-1}. $$
The involution $\theta$ acts by $\pm \Id$ on all the blocks of the partition \eqref{eq:matrixE7}. We describe $\theta$ by saying the sign on each block.
	$$\theta\equiv \begin{pmatrix}[cccccc]
		\multicolumn{2}{c}{\multirow{2}{*}{$+$}} & \multicolumn{2}{c}{\multirow{2}{*}{$-$}} & \multirow{2}{*}{$+$} & \multirow{2}{*}{$-$}  \\
		& & & & \\ 
		\multicolumn{2}{c}{\multirow{2}{*}{$-$}} & \multicolumn{2}{c}{\multirow{2}{*}{$+$}} & \multirow{2}{*}{$-$} & \multirow{2}{*}{$+$}  \\
				& & & &\\ 
	    \multicolumn{2}{c}{+}  & \multicolumn{2}{c}{-} & + & -\\
	    	    \multicolumn{2}{c}{-} & \multicolumn{2}{c}{+} &- & +
	\end{pmatrix}.$$
	
We show now that the element
$$
g=\begin{pmatrix}[cccccc]
	\multicolumn{2}{c}{\multirow{2}{*}{$0$}} & \multicolumn{2}{c}{\multirow{2}{*}{$-1$}} & \multirow{2}{*}{$0$} & \multirow{2}{*}{$0$}  \\
	& & & & \\ 
	\multicolumn{2}{c}{\multirow{2}{*}{$1$}} & \multicolumn{2}{c}{\multirow{2}{*}{$0$}} & \multirow{2}{*}{$0$} & \multirow{2}{*}{$0$}  \\
	& & & &\\ 
	\multicolumn{2}{c}{0}  & \multicolumn{2}{c}{0} & 0 & -1\\
	\multicolumn{2}{c}{0} & \multicolumn{2}{c}{0} &1 & 0
\end{pmatrix}
$$
belongs to $\cA_\theta\subset E_7$ and $\Ad g\notin \Ad H$.		  Indeed, for $P=(X,Y,\xi,\eta)$ and $Q=(Z,W,\zeta,\omega)$ we have
	$$gP = (-Y,X,-\eta,\xi),\qquad gQ=(-W,Z,-\omega,\zeta).$$
	By using \eqref{eq:Phi-E7} and the properties $X\times Y=Y\times X$ and $Y\lor X = \phantom{.}^t(X\lor Y) $ for $X$, $Y\in\fJ$ (see \cite[Lem. 3.4.3]{yokota}), we have 
			$$\al P \times \al Q = \Phi(-^t\phi, -B,-A,-\nu),$$
which satisfies $\al (P\times Q) = (\al P \times \al Q) \al$. Finally, we cannot have $\Ad g= \Ad h$ for some $h\in H$, as this would imply $h=gz$ for $z\in Z(E_7)$, which is not possible as $\pm g\notin H$.

		$\bullet$ $(E_7,A_7)$. Denote by $\tau$ the conjugation on a complexification: in $\C$ is the usual conjugation, whereas in $\Oc_\C$, we have $\tau(\al + i\be)=\al-i\be$ for $\al$, $\be\in\Oc$. Denote by $\gamma$ the involution of $\Oc_\C$ given, using $\Oc_\C=\HH_\C+\HH_\C l$,  by $$\gamma(x+ y l)=x-y l,$$ where  $x$, $y\in \HH_\C$.  Define a linear involution $\tau \gamma:\fB\to\fB$ by
		$( X, Y,\xi,\eta)\mapsto (\tau \gamma X,\tau \gamma Y,\tau \xi,\tau \eta)$ and consider the involution $\theta:E_7\to E_7$ given by
		$$		\theta(\alpha)=\tau\gamma\alpha\gamma\tau. $$
		The involution $\tau\gamma:\Oc_\C\to\Oc_\C$ is given, for $a+bi$, $c+di\in \HH_\C$, by
		$$ \tau\gamma:  (a+bi)+(c+di)l  \mapsto (a-bi)-(c-di)l. $$
		Its $\pm 1$-eigenspaces are, respectively,
		$$ R:= \HH + i\HH l, \qquad S:= i\HH + \HH l.$$
		Note that $\tau a=-a\tau$ for $a\in \C$ implies $a\in i\R$. On the other hand, for $p\in\fJ$, the condition $\tau \gamma p=-p\tau\gamma$ acting on $\C$ implies $p\in S$. Moreover, $\tau\gamma A = A\tau\gamma$ if and only if $A$ interchanges $R$ and $S$. By these arguments we get that for $g$ to satisfy $\theta(g)=-g$, it is a necessary condition that, with the notation of \eqref{eq:matrixE7}, 
		\begin{itemize}
		 \item $A,B,C,D$ interchange $R$ and $S$,
		 \item $k,l,p,q\in S$,
		 \item $r,s,u,v$ are given by inner product with an element of $S$,
		 \item $a,b,c,d\in i\R$.
		\end{itemize}
        Alternatively, to have $\theta(g)=g$ we must have that $A,B,C,D$ send $R$ to $R$ and $S$ to $S$, together with $k,l,p,q\in R$, $r,s,u,v$ are given by inner product with an element of $R$, and $a,b,c,d\in \R$.
		
		We see now that the element   
$$
g=\begin{pmatrix}[cccccc]
\multicolumn{2}{c}{\multirow{2}{*}{$0$}} & \multicolumn{2}{c}{\multirow{2}{*}{$i$}} & \multirow{2}{*}{$0$} & \multirow{2}{*}{$0$}  \\
& & & & \\ 
\multicolumn{2}{c}{\multirow{2}{*}{$i$}} & \multicolumn{2}{c}{\multirow{2}{*}{$0$}} & \multirow{2}{*}{$0$} & \multirow{2}{*}{$0$}  \\
& & & &\\ 
\multicolumn{2}{c}{0}  & \multicolumn{2}{c}{0} & 0 & i\\
\multicolumn{2}{c}{0} & \multicolumn{2}{c}{0} &i & 0
\end{pmatrix}
$$
belongs to $\cA_\theta\subset E_7$  and $\Ad g\notin \Ad H$.		 Indeed, for $P=(X,Y,\xi,\eta)$ and $Q=(Z,W,\zeta,\omega)$ we have
$$\al P = (iY,iX,i\eta,i\xi),\qquad \al Q=(iW,iZ,i\omega,i\zeta).$$
By using \eqref{eq:Phi-E7} and the properties $X\times Y=Y\times X$ and $Y\lor X = \phantom{.}^t(X\lor Y) $ for $X$, $Y\in\cJ$, we have 
$$\al P \times \al Q = \Phi(-^t\phi, B,A,-\nu),$$
which satisfies $\al\circ (P\times Q) = (\al P \times \al Q)\circ \al$. Finally, we cannot have $\Ad g= \Ad h$ for some $h\in H$, as this would imply $h=gz$ for $z\in Z(E_7)$, which is not possible as $\pm g\notin H$.
\end{proof}

\subsection{Nice symmetric pairs and summary of results}
\label{sec:nice-summary}

We combine the notion of a pleasant symmetric pair with the notion of a nice symmetric pair. When the Aizenbud-Gourevitch criterion was introduced, the main tool to prove regularity was speciality \cite[Def. 7.3.4]{ag-duke}, which was in turn proved by looking at the so-called negative distinguished defect (see \cite[Sec. 5]{aizenbud-13} for more details). Being of negative distinguished defect is equivalent to being a nice symmetric pair as studied by Sekiguchi (although so named in \cite{levasseur-stafford}). Niceness depends only on the symmetric Lie algebra pair and Sekiguchi classified all such pairs.

\begin{lemma}[\cite{sekiguchi}]
	The nice simple symmetric pairs are exactly:
	\begin{align*}
	&(A_{2n},B_n),&& (A_{2n-1},D_n), &&(A_{2r+1},A_r+A_r+\C),&&(C_{n},A_{n-1}+\C),\\
	&(B_{2r},B_r+D_r),&& (B_{2r+1},D_{r+1}+B_r),&&(D_{2r+1},D_{r}+D_{r+1}),  && (D_{2r},D_{r}+D_{r}),\\ & (D_{2r+1},B_{r}+B_{r}), && (D_{2r+2},B_{r}+B_{r+1}),&&(G_2,A_1+A_1),&&(F_4,C_3+A_1),\\ & (E_{6},C_4),
	&&(E_6,A_5+A_1),&&(E_{7},A_7),&&(E_8,D_8).
	\end{align*}
\end{lemma}

\begin{lemma}[\cite{aizenbud-13}, Cor. 5.2.7]
	Nice symmetric pairs are regular.
\end{lemma}

%
%$(A_{2n},B_n)$,  $(A_{2n-1},D_n)$, $(A_{2r+1},A_r+A_r+\C)$, $(C_{n},A_{n-1}+\C)$, $(B_{2r},B_r+D_r)$, $(D_{2r},D_{r}+D_{r})$, $(D_{2r+1},D_{r}+D_{r+1})$, $(D_{2r},B_{r}+B_{r})$, $(D_{2r+1},B_{r}+B_{r+1})$, $(G_2,A_1+A_1)$, $(F_4,C_3+A_1)$, 	$(E_{6},C_4)$, 	$(E_6,A_5+A_1)$, $(E_{7},A_7)$ and $(E_8,D_8)$.

By combining these two lemmas with the results in Sections \ref{sec:classical-pleasant} and \ref{sec:spin-pairs} and Theorem \ref{theo:exceptional-pleasant}, we have Tables \ref{tab:pleasant-nice},  \ref{tab:pleasant-nice2} and \ref{tab:pleasant-nice3}. In the column ``pleasant'', \xmark\, means that only the adjoint form is pleasant, \cmark\, that the simply connected group is pleasant, and if a group appears, it is the group such that all its quotients, including itself, are pleasant. The notation $\SL/\Z_{2^{max}}$ means killing the maximal even part of the centre.

\begin{table}[!htpb]
	\vspace{0cm}
	\begin{center}
		\setlength{\tabcolsep}{2pt}
		\begin{tabular}{cc|c|c}
 & & pleasant & nice \\ \hline \hline
 $(A_{2n},B_n)$ &  & \cmark  & \cmark \\ \hline
 $(A_{2n-1},D_n)$ & & \footnotesize $\SL/\Z_{2^{max}}$ &  \cmark\\ \hline 
 $(A_{2n-1},C_n)$ & &  \cmark  &  \xmark \\ \hline 
 \multirow{2}{*}{$(A_{r+s+1},A_r+A_s+\C)$} & \footnotesize $r\neq s$ & \cmark  & \xmark \\
 & \footnotesize  $r=s$ & \footnotesize $\SL/\Z_{2^{max}}$   & \cmark \\ \hline
 $(C_{n},A_{n-1}+\C)$ & & \xmark & \cmark \\ \hline 
  \multirow{2}{*}{$(C_{r+s},C_r+C_s)$} & \footnotesize $r\neq s$ & \cmark & \xmark \\
 & \footnotesize  $r=s$ &  \xmark  & \xmark \\ \hline
\end{tabular}\qquad 
\end{center}
\vspace{2pt}
\caption{Pleasant and nice linear and symplectic pairs.}\label{tab:pleasant-nice}
%($\fsl$, $\fsp$)
\vspace{0cm}
\end{table}

\begin{table}[!htpb]
		\begin{center}
		\setlength{\tabcolsep}{2pt}
\begin{tabular}{cc|c|c}
	 & & pleasant & nice \\ \hline \hline
	\multirow{3}{*}{$(B_{r+s},B_r+D_s)$} & \footnotesize $r\neq s-1,s$ & \footnotesize $\SO$ & \xmark \\
		& \footnotesize  $r=s-1$ & \footnotesize $\SO$  & \cmark \\ 
	& \footnotesize  $r=s$ & \xmark & \cmark \\ \hline
			$D_{2q}$ & & & \\ 
	\multirow{3}{*}{\parbox{100pt}{\center $(D_{r+s},D_{r}+D_s)$\\ $(D_{r+s+1},B_{r}+B_s)$}} & \footnotesize $r\neq s-1,s$ & \footnotesize$\SO$ & \xmark \\
	& \footnotesize  $r=s-1$ & \footnotesize $\SO$  & \cmark \\ 
	& \footnotesize  $r=s$ & \xmark & \cmark \\ \hline
	 $D_{2q+1}$ & & & \\ 
	\multirow{3}{*}{\parbox{100pt}{\center $(D_{r+s},D_{r}+D_s)$}} & \footnotesize $r\neq s-1,s$ &  \footnotesize $\SO$ & \xmark \\
	& \footnotesize  $r=s-1$ &  \footnotesize $\SO$  & \cmark \\ 
	& \footnotesize  $r=s$ & \xmark & \cmark \\ \hline
	 $D_{2q+1}$ & & & \\ 
	\multirow{3}{*}{\parbox{100pt}{\center  $(D_{r+s+1},B_{r}+B_s)$}} & \footnotesize $r\neq s-1,s$ & \cmark & \xmark \\
	& \footnotesize  $r=s-1$ &  \cmark  & \cmark \\ 
	& \footnotesize  $r=s$ & \xmark & \cmark \\ \hline
	$(D_r,A_{r-1}+\C)$ &  & \xmark  & \xmark \\ \hline
\end{tabular}\qquad 
\end{center}
\vspace{2pt}
\caption{Pleasant and nice orthogonal and Spin pairs.}\label{tab:pleasant-nice2}
\vspace{0cm}
\end{table}
% ($\fso$)
% \multicolumn{2}{c|}{\footnotesize $D_{even}$}
% \multicolumn{2}{c|}{\footnotesize $D_{odd}$}

\begin{table}[!htpb]
	\begin{center}
		\setlength{\tabcolsep}{2pt}
\begin{tabular}{c|c|c}
	 & pleasant & nice   \\ \hline \hline
	$(G_2,A_1+A_1)$  & \cmark  & \cmark \\ \hline
	$(F_4,B_4)$  &\cmark  &  \xmark\\ 
	$(F_4,C_3+A_1)$ &\cmark  & \cmark \\ \hline 
	$(E_{6},C_4)$ & \cmark & \cmark \\ 
	$(E_6,A_5+A_1)$ &  \cmark & \cmark \\ 
	$(E_6,F_4)$ & \cmark  & \xmark \\
	$(E_6, D_5+\C)$ & \cmark   & \xmark \\ \hline
	$(E_{7},A_7)$ & \xmark  & \cmark \\
	$(E_7,D_6+A_1)$& \cmark  & \xmark \\ 
	$(E_7,E_6+\C)$ & \xmark  & \xmark \\ \hline
	$(E_8,D_8)$ & \cmark   & \cmark \\ 
	$(E_8,E_7+A_1)$ &  \cmark  & \xmark \\ \hline
\end{tabular}
\end{center}
\vspace{0pt}
\caption{Pleasant and nice exceptional pairs.}\label{tab:pleasant-nice3}
\end{table}

Note that this does not mean that $(C_{2r},C_r+C_r)$, $(D_r,A_{r-1})$,  $(\Spin_{m+n},\Spin_m\times_{\Z_2} \Spin_n))$ with $|m-n|>2$, apart from $m+n=4q+2$ with $m,n$ odd, and $(E_7,E_6+\C)$ are not regular. It says, though, that other techniques are needed to possibly show their regularity.

% regularity of quotient of symmetric pairs, not so clear, as Ad g admissible for H' does not seem to imply admissibility for H (an H'-orbit splits into several H-orbits)
%
% \subsection{Products and non-connected groups.}
%
% PRODUCTS OF REGULAR, OR AT LEAST PRODUCTS OF PLEASANT AND NICE, OR MENTION BOTH
%
% product an quotient symmetric pair of nice symmetric pairs is nice. Also for p
%
% Note that the product of regular pairs is regular, by \cite[Prop. 7.4.4]{ag-duke}, but it is not so clear whether the quotient symmetric pair of a regular pair is regular. CHECK the reason for this is that one could potentially have a non-admissible element $g$ for $H$, such that $\pi(g)$ is admissible for the bigger group $H'$.
%
% TALK ABOUT REGULARITY OF NON-CONNECTED PAIRS, AS THEY APPEAR POTENTIALLY AS DESCENDANTS?.

\section{Descendants of complex symmetric pairs}
\label{sec:visual}

Descendants of symmetric pairs are centralizers (in the sense of Definition \ref{def:descendant}) of semisimple elements  
$$ x\in P:= \{ g\theta(g)^{-1} \st g\in G \}. $$
The aim of this section is to describe $(G_x,H_x,\theta_{|G_x})$ by means of the Satake diagram of $(G,H,\theta)$. We start with some facts about centralizers of semisimple elements.

\subsection{Centralizers of semisimple elements and extended Dynkin diagrams}
\label{sec:centralizers}

The results of this section are valid for algebraically closed fields, but again we use $\C$. The following proposition follows from \cite[Ch. II]{springer-steinberg} (see  also \cite[Ch. 2]{humphreys}).
% Humphreys $\S 2.2$, $2.12$
% It originally comes from \cite[8.1]{steinberg}

\begin{proposition}
	Let $x$ be a semisimple element of a connected  semisimple group $G$. Take any maximal torus $T$ containing $x$ and consider the root system $\Delta(G,T)$ with root groups $U_\al$. Then, the centralizer $G_x$ is the subgroup generated by $T$, those root groups $U_\al$ for which $\al(x)=1$, and the elements $n_w\in N(T)$ from a choice of Weyl group representatives commuting with $x$. The centralizer $G_x$ is therefore reductive but not necessarily connected. If $G$ is simply connected, the centralizer $G_x$ is moreover connected. 
\end{proposition}
 
Consequently, as $G_x$ and $G$ share a maximal torus $T$, we have an inclusion
$$\Delta_x:=\Delta(G_x,T)\subseteq \Delta,$$ 
where $\Delta$ denotes $\Delta(G,T)$. It is not always possible to choose a set of simple roots $\Pi_x\subset  \Delta_x$ extending to a set of simple roots for $\Delta$, but it is indeed possible to extend $\Pi_x$ to $\Delta\cup \{ \gamma\}$, where $\gamma$ is the lowest root for the ordering given by $\Delta$ (see Proposition \ref{prop:subdiagram} for the precise statement). This result can be read off from \cite[\S 14.1]{gorenstein-lyons}, which in turn relies on the statement of Problem \S 4-4d) in \cite[Ch. VI]{bourbaki456}, so it is a better idea to show a direct argument here, as we will have to build on it in Section \ref{sec:visual-symmetric}.

% (see also \cite[\S 5]{dynkin} and \cite{borel-siebenthal})

We shall use the affine Weyl group, the Weyl alcoves and their properties (see \cite[Sec. VI.\S 2]{bourbaki456} for more details and proofs). The affine Weyl group 
$$ W_a(\Delta):= W(\Delta) \ltimes Tr(\Delta^\vee)$$ is the semidirect product of the Weyl group $W(\Delta)$ and the group $Tr(\Delta^\vee)$  generated by translations by coroots $\alpha^\vee=2\frac{\al}{\la \al,\al \ra}$ for any $\al\in\Delta$. For the vector space $\R\Delta$  with inner product $\la\, ,\, \ra$, the Weyl alcoves are the connected components in $\R\Delta$ of the complement of the hyperplanes, for $\alpha \in \Delta$ and $k\in \Z$, 
$$ L_{\alpha,k}= \{ v\in \R\Delta \st \la v,\alpha\ra = k \},$$ 
whereas the Weyl chambers are the connected components in $\R\Delta$ of the complement of the hyperplanes $L_{\alpha,0}$ for $\alpha\in\Delta$.

\begin{lemma}\label{lemma:simply-transitive}
 The affine Weyl group permutes the hyperplanes $L_{\alpha,k}$ and acts simply transitively on the Weyl alcoves, whereas the Weyl group permutes the hyperplanes $L_{\alpha,0}$ and acts simply transitively on the Weyl chambers.
\end{lemma}

We consider the closure of the union of the Weyl alcoves whose closure contains zero, 
\begin{equation}
A_0:=\{ v\in\R\Delta \st \la v,\alpha\ra \leq 1 \textup{ for all } \al \in \Delta \}. 
\end{equation}
The results that we shall use are a consequence of Lemma \ref{lemma:simply-transitive}.

\begin{lemma}\label{lemma:translation-in-Wa}
 For $v\in \R\Delta$, there is $t\in Tr(\Delta^\vee)$ such that $v+t\in A_0$.
\end{lemma}
% The element $t$ is moreover unique if and only if $\la v,\al\ra < 1$ for all $\al\in\Delta$

\begin{lemma}\label{lemma:Weyl-chamber-alcove}
 Inclusion gives a bijective correspondence between Weyl chambers and Weyl alcoves contained in $A_0$. The boundary of a Weyl alcove is contained in the boundary of the corresponding Weyl chamber and the hyperplane $L_{-\gamma,1}$ where $-\gamma$ is the highest root with respect to the order determined by the Weyl chamber.
\end{lemma}
{\noindent We include an example for the root system $B_2$ in Figure \ref{fig:Weyl-alcoves} for the sake of clarity.}

\begin{figure}[h] 
\begin{tikzpicture}[>=latex]
\foreach \k in {-2,...,2}{
    \draw[blue,dashed] (-2,\k) -- (2,\k);
    \draw[blue,dashed] (\k,-2) -- (\k,2);
    }

\foreach \j in {-2,2} {
    \foreach \k in {-2,...,1}{
    \draw[blue,dashed] (\j,\k) -- (-\k,-\j);
    \draw[blue,dashed] (\j,\k) -- (\k,\j);
}
}  
\draw[blue,dashed] (0,-2) -- (0,-2.5);  
\draw[blue,dashed] (0,2) -- (0,2.5);  
\draw[blue,dashed] (2,2) -- (2.4,2.4);  
\draw[blue,dashed] (-2,-2) -- (-2.4,-2.4);
\draw[blue,dashed] (2,-1) -- (2.5,-1.5);  
\draw[blue,dashed] (-1,2) -- (-1.5,2.5);
\fill[blue!25] (1,0) -- (0,1) -- (-1,0) -- (0,-1) -- cycle;
\fill[blue!75] (0,0) -- (0,1) -- (.5,.5) -- cycle;
\draw[->] (0,0) -- (1,0);
\draw[->] (0,0) -- (1,1);
\draw[->] (0,0) -- (0,1);
\draw[->] (0,0) -- (-1,1);
\draw[->] (0,0) -- (-1,0);
\draw[->] (0,0) -- (-1,-1);
\draw[->] (0,0) -- (0,-1);
\draw[->] (0,0) -- (1,-1);
\node at (0.55,-0.15) {\small $\alpha_1$};
\node at (-0.5,0.8) {\small $\alpha_2$};
\node at (0.45,2.25) {\small $L_{\alpha_1,0}$};
\node at (2.6,2.05) {\small $L_{\alpha_2,0}$};
\node at (2.9,-1.1) {\small $L_{2\alpha_1+\alpha_2,1}$};
\end{tikzpicture}
\caption{Weyl alcoves for $B_2$, with $A_0$ shaded and the alcove given by the set $\{\al_1,\al_2\}$, with highest root  $2\al_1+\al_2$, shaded darker.}
\label{fig:Weyl-alcoves}
\end{figure}
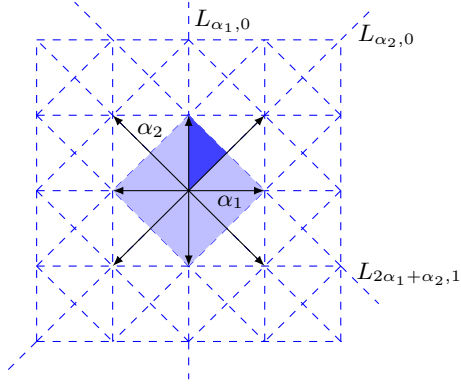

The statement and the proof of the following result will be important in what follows. Note that a set of simple roots $\Pi_x$ for $\Delta_x$ must satisfy, for  $\alpha,\beta\in \Pi_x$, that $\alpha-\beta\not\in \Pi_x$, and the set $\Pi\cup \{\gamma\}$ is maximal with this property.

\begin{proposition}\label{prop:subdiagram}
	Let $x$ be a semisimple element of $G$. There exists an ordering of $\Delta$ with corresponding set of simple roots $\Pi\subset \Delta$ and lowest root $\gamma$, such that
\begin{equation}\label{eq:Pix}
\Pi_x := \Delta_x \cap ( \Pi\cup \{\gamma\} ) 
\end{equation}
	is a set of simple roots for $\Delta_x$.
\end{proposition}

\begin{proof}
 Evaluation on $x$ defines a map from the root lattice $\Z \Delta$ to $\C^*$ with kernel the lattice $\Z\Delta_x$. The first step is to find $v\in \R \Delta$ such that
 \begin{equation}\label{eq:ZDeltax}
	 \Z \Delta_x = \{ \al \in\Z \Delta \st \la v, \al\ra \in \Z\}. 	  
 \end{equation} Choose a basis $B=\{b_1,\ldots,b_r\}$ of the lattice $\Z\Delta_x$, and complete it with a set $\{c_{r+1},\ldots,c_n\}\subset \Z\Delta$ to a basis of $\R\Delta$. Define an element $v\in \R\Delta$ by  
\begin{align}\label{eq:defining-v}
\la v, b_j \ra &  =  1, & 
\la v, c_j \ra & = \lambda_j,  
\end{align}
for $\{\lambda_j\}\subset \R\setminus \Q$ linearly independent over $\Q$, so that \eqref{eq:ZDeltax} is satisfied.
%where the last equality follows from the choice of $\{\lambda_j\}$.

Secondly, by Lemma \ref{lemma:translation-in-Wa}, we can assume that $v$ satisfying \eqref{eq:ZDeltax} lies on $A_0$. Indeed, the addition of a coroot to $v$, and hence of an element in $Tr(\Delta^\vee)$, does not change the condition  $\la v, \al\ra \in \Z$ for any $\al\in \Delta$.

%% ALMA

The element $v$ lies in some Weyl chamber or in the boundary of some of them. By choosing a neighbouring Weyl chamber in the latter case, we get an order and a set of simple roots $\Pi=\{\al_i\}$. The positive roots $\al$ satisfy $$\la v,\al \ra \geq 0,$$
whereas the lowest root $\gamma$ and highest root $-\gamma$ with respect to this ordering satisfy
\begin{equation}\label{eq:inequality-lowest-root}
 \la v, \gamma \ra \geq -1, \qquad\qquad \la v,-\gamma\ra \leq 1.
\end{equation}

Finally, if \eqref{eq:inequality-lowest-root} are strict inequalities, by \eqref{eq:ZDeltax}, the roots in $\Delta_x$ are those $\al$ such that $\la v,\al \ra = 0$, that is,  suitable positive or negative linear combinations of the elements of $$\Pi_x = \{ \al_i \in \Pi \st \la v,\al_i\ra =0\}.$$ On the other hand, if equality is reached in \eqref{eq:inequality-lowest-root}, the roots in $\Delta_x$ are those $\al$ such that $\la v,\al \ra = 0, \pm 1$. In this case, these are suitable positive or negative linear combinations of 
$$ \Pi_x=\{ \al_i \in \Pi \st \la v,\al_i\ra =0\}\cup \{\gamma\},$$
and the result follows.

\end{proof}

The Dynkin diagram of $G_x$ can then be regarded inside the extended Dynkin diagram of $G$, which is obtained by adding to the Dynkin diagram of $G$ an extra node for the root $\gamma$ (and connecting it to the other nodes following the usual rules of angle they form, and length ratio). We show the extended Dynkin diagrams for simple groups in Table \ref{tab:extended}, where the white node corresponds to the lowest root $\gamma$.

%  as they will be used in Section \ref{sec:visual-symmetric}

 \begin{table}[h!]
	\begin{center}
		\begin{tabular}[320pt]{ccccc}
			\hline \hline
			\parbox[t]{30pt}{ \vspace{9pt}  \center \dynkin[extended]{A}{1}} & \parbox[t]{70pt}{\center  \dynkin[extended]{A}{}}& \parbox[t]{70pt}{\center \dynkin[extended]{B}{}}  & \parbox[t]{70pt}{ \vspace{9pt} \center \dynkin[extended]{C}{}} & \parbox[t]{70pt}{\center \dynkin[extended]{D}{}}    \\
			\footnotesize $A_1$ & \footnotesize $A_n$ & \footnotesize $B_n$ & \footnotesize $C_n$ & \footnotesize $D_n$ \\
			& \\ \hline
		\end{tabular}
		\begin{tabular}[320pt]{ccccc}
			\parbox[t]{30pt}{ \vspace{20pt}  \center \dynkin[extended]{G}{2}} & \parbox[t]{60pt}{  \vspace{20pt} \center  \dynkin[extended]{F}{4}}& \parbox[t]{60pt}{\center \dynkin[extended]{E}{6}}  & \parbox[t]{80pt}{ \vspace{10pt} \center \dynkin[extended]{E}{7}} & \parbox[t]{80pt}{\vspace{10pt} \center \dynkin[extended]{E}{8}}    \\
			\footnotesize $G_2$ & \footnotesize $F_4$ & \footnotesize $E_6$ & \footnotesize $E_7$ & \footnotesize $E_8$ \\
			& \\ \hline \hline &\\
		\end{tabular}
	\end{center}
	\caption{Extended Dynkin diagrams for simple groups.}\label{tab:extended}
\end{table}

Note that the previous proposition is giving the Dynkin diagram and hence the Lie type of the semisimple part of the centralizer. The abelian part has rank equal to $\rk G - \rk G^{ss}_x$. At the level of groups, one has a finite quotient of a product of a semisimple and an abelian group.

\begin{example}
	The Lie type of the centralizer of a semisimple element in the Lie group $\SO_{2n}$ is $$ \SO_{2r} + \SL_{s_1+1} + \ldots \SL_{s_l+1} + \SO_{2t} + \C^u,$$
	with $2r+\sum_{i=1}^l (s_i+1) +2t + u= 2n$. 
\end{example}

\subsection{Computation of descendants}
\label{sec:visual-symmetric}

We introduced the Satake diagram in Section \ref{sec:Satake} by means of a maximal $\theta$-split torus $A$. We want to use it now to extract information about the  descendants. We first recall two results.

\begin{lemma}[\cite{richardson}, \S 7.5]\label{lem:richardson}
	Let $A$ be a fixed maximal $\theta$-split torus. An element $x\in P=\{ g\theta(g)^{-1} \st g\in G \}$ is semisimple if and only if $H_0\cdot x$ meets $A$.
\end{lemma}

\begin{lemma}[\cite{vust}, Cor. 5]\label{lem:vust}
	The group $H_0$ acts transitively on maximal $\theta$-split  tori.
\end{lemma}

We consider the extended Satake diagram.

\begin{definition}
	The extended Satake diagram of $(G,H,\theta)$ is the diagram resulting from colouring and adding arrows to the extended Dynkin diagram of $G$ following the rules of the Satake diagram of $(G,H,\theta)$.
\end{definition}

\begin{lemma}
	The extra node of the extended Dynkin diagram of a (non-trivial) symmetric pair remains white and there are no bars on the extra node.
\end{lemma}

\begin{proof}
	The fact that the node is white follows from the fact that all simple roots contribute to the lowest root and the roots are independent. Moreover, the involution sends the lowest root with respect to a $\theta$-basis to itself, as it is unique and the action of $-\theta$ permutes the simple roots.
\end{proof}

\begin{remark}
	If we consider the possibilities of Remark \ref{rem:Satake-extreme}, we have that when the involution of the Satake diagram is trivial (all nodes are black), the extra node is also black. On the other hand, in the case of $(G\times G,\Delta G)$ for $G$ a simple group, there are two extra nodes, which are connected by a bar.
\end{remark}

The following result is fundamental to this work.

\begin{theorem}\label{theo:visual-descendants}
	The Satake diagrams of descendants $(G_x,H_x,\theta_{|G_x})$ for semisimple $x\in P$ are exactly the (possibly disconnected) Satake diagrams obtained by erasing at least one white node (and its incident edges) from the extended Satake diagram of $(G,H,\theta)$ in such a way that nodes connected by a bar are both kept or erased.
\end{theorem}

\begin{proof}
We split the proof into three steps.

\textbf{Step 1: finding a suitable $\theta$-basis.} Let $x\in P$ be semisimple. As a consequence of Lemmas \ref{lem:richardson} and \ref{lem:vust}, we can find a maximal $\theta$-split torus $A$ such that $x\in A$. Choose a maximal $\theta$-stable torus $T$ containing $A$.  We want to choose a $\theta$-basis  (Definition \ref{def:theta-basis}) satisfying \eqref{eq:Pix} in Proposition \ref{prop:subdiagram}, of which we use the proof. Consider the element $v\in A_0$ defined by \eqref{eq:defining-v} and the action of the affine Weyl group. The key point is to now choose the neighbouring Weyl chamber appropriately.

For any Weyl chamber whose closure contains $v$, we check that the corresponding set of simple roots $\Pi$ satisfies, for all $\al\in\Pi$, 
$$ 0\leq \la v,\al \ra < 1.$$
Being greater or equal than zero follows from $v$ being in the closure of the Weyl chamber, while being less or equal than one, from $\la v,-\gamma\ra\leq 1$, where $-\gamma$ is the highest root. Moreover, if $\la v,\be \ra=1$ for some $\be\in\Pi$, then $\la v,\al\ra=0$ for $\al\neq \be$, as $\la v,-\gamma\ra\leq 1$, and then $\Delta_x$ would be the whole $\Delta$.

We have that $\al\in \Delta_x$ if and only if $-\theta\alpha\in \Delta_x$. Hence, for $\al\in\Pi$, $\la v,\al\ra >0$ if and only if $\la v,-\theta \al\ra >0$, so both $\al$ and $-\theta\al$, when different, are  positive. On the other hand, if $\la v,\al \ra =0$, then $\la v,-\theta \al \ra =0$, so, in the case where $\al $ and $-\theta\al$ are two different simple roots, we will choose the neighbouring Weyl chamber in such a way that all these pairs $\al$, $-\theta \al$ are positive.

This gives a set of simple roots such that $\al>0$ and $-\theta\al>0$ for $\al\in\Pi$, so we have that $\al>0$ implies $-\theta \al>0$ for $\al\in\Delta$, the condition for being a $\theta$-basis.

\textbf{Step 2: necessary conditions on the extended Satake diagram.} Consider the Satake diagram with respect to this choice of $\theta$-basis. Just as in the proof of Proposition \ref{prop:subdiagram}, the simple roots $\al\in \Pi$ such that  $\la v, \al \ra = 0$ and possibly the lowest root $\gamma$ when $\la v, \gamma \ra = -1$ become a set of simple roots for $G_x$. As discussed after Proposition \ref{prop:subdiagram}, the Dynkin diagram of the centralizer of $x$ is given by a subdiagram of the extended Dynkin diagram, that is, a diagram obtained by erasing nodes.
Since $x$ is in $A$, we have that $\alpha(x)=1$ for those $\alpha\in\Pi$ corresponding to black nodes, and hence we cannot erase any of the black nodes. Moreover, if the white node for a simple root $\al$ is erased, $\alpha(x)\neq 1$, as
$$ (-\theta\alpha)(x)=(\alpha(\theta(x)))^{-1}=\alpha(x)\neq 1,$$
the node $-\theta\alpha$, in case it is different and hence connected by an arrow, is also erased.

% otherwise $$\alpha(g\theta(g^{-1})=...$$

We show next that the colouring and the bars stay the same after erasing some white nodes.  The involution on $G_x$ is the restriction of the involution on $G$, and the groups $G$ and  $G_x$ share the same maximal torus, but the semisimple part of $G_x$ is possibly smaller, so, in principle, white nodes could potentially become black or connected by a bar. We show that this is not the case.

Consider first the case in which we erase some nodes from the Dynkin diagram (without extending it). The roots $\Pi=\{\al_i\}_{1}^r$ corresponding to the nodes of the diagram give a finite-covering map 
\begin{equation}\label{eq:covering-map-T}
(\al_1,\ldots,\al_r):T\to (\C^*)^r.
\end{equation}
By considering the roots corresponding to the white nodes and choosing only one in case they are connected with a bar (say, by reordering if necessary, $\{\al_i\}_1^{s}$, where $s$ is the rank of $A$), we get a finite-covering map 
\begin{equation}\label{eq:covering-map-A}
(\al_1,\ldots,\al_{s}):A\to (\C^*)^{s}.
\end{equation}
When looking at a connected component of the Dynkin diagram of the semisimple part of $G_x$, we have a corresponding torus $T'\subset T$, and $A'=T'\cap A$. The roots corresponding to $A'$ are a subset $\{\al_i\}_1^{s'}\subset \{\al_i\}_1^{s}$, where $s'$ is the rank of $A'$. A white node would turn black when its corresponding root vanished on $A'$, but this would mean that the kernel of this root is not finite, which contradicts \eqref{eq:covering-map-A} being a finite-covering map. Thus, a white node cannot turn black. Analogously, no white node can become connected to any other root, as this would again imply that \eqref{eq:covering-map-A} is not a finite-covering map.

For the general case, any choice of all but one white node, say $\alpha_j$, of the extended Dynkin diagram gives such covering maps, like \eqref{eq:covering-map-A},  for the maximal $\theta$-split torus $A$ of $G$,
$$(\alpha_1,\ldots,\widehat{\alpha_j},\ldots,\alpha_r,\gamma): T\to (\C^*)^r,$$ where $\widehat{\alpha_j}$ denotes that the root $\alpha_j$ is missing. So no white root can become black or connected to other root when erasing the first node. When more roots are erased, we are in the situation of the previous paragraph.

\textbf{Step 3: all suitable Satake diagrams are realizable.} 	Finally, we show that all Satake diagrams where white nodes have been erased from the extended Satake diagrams are actually realizable as the semisimple part of some $G_x$ for $x$ semisimple in $A$. We will do this by actually describing an element $x\in A$. In the lowest root $\gamma$, every simple roots appears with multiplicity less or equal to $-1$. The structure of $G_x$, that is, the roots $\beta$ such that $\beta(x)=1$, depends only on the values of $\alpha_i(x)$. By the surjectivity of the maps \eqref{eq:covering-map-T} and \eqref{eq:covering-map-A}, we will always find an element $x\in T$ for any choice of values for $\{\alpha_i(x)\}$. We set these values to be $1$ both for the roots of the black nodes and the white nodes that we preserve, and the same $z$ for the white nodes we erase. The value of $\gamma(x)$ is determined by these choices: we have $\gamma(x)=z^{-m}$ for some $m>0$ (the case $m=0$ means that $G_x=G$).
	If we want to erase the node corresponding to the lowest root, we can take $z$ to be an $m+1$-primitive root of unity. If we want to keep  the white node corresponding to the lowest root, we take $z$ to be an $m$-primitive root of unity. The fact that $z$ is an $m$ or $m+1$-primitive root assures that the roots that stay are a set of simple roots. Indeed, for $\al\in\Pi$, the value $\alpha(x)$ ranges from $z^{-m}$ to $z^{m}$. When we are erasing the lowest root ($z$ is an $m+1$-primitive root), the only option for $\alpha(x)$ to be $1$ is for $\alpha$ to be a positive or negative linear combination of the roots in $\Pi$ which are $1$ on $x$, exactly those nodes that we are keeping. When we are keeping the lowest root ($z$ is an $m$-primitive root), $\alpha(x)$ is $1$ when it is a positive or negative linear combination of the lowest root and the simple roots that are $1$ on $x$, as any other root will have value $z^{t}$ with $-m<t<m$ and $t\neq 0$. A final check is needed: when $m=1$, which is only possible when a simple root appears with multiplicity $-1$ in the lowest root (that is, in types $A$, $B$, $C$, $D$, $E_6$ and $E_7$), keeping the lowest root means setting $z=1$, which would mean keeping everything. Still, the statement of the theorem still holds, as it can be verified case by case: the Dynkin diagram resulting from erasing a simple root with multiplicity $-1$ from the extended Dynkin diagram is actually the starting Dynkin diagram (but in a different position). We show here the possible cases, with a double node for the lowest root and the other nodes white for simplicity (other suitable colourings can be read from Table \ref{tab:Satake}). Note that for $A_n$ and $E_6$ more than one node appears with multiplicity $-1$ but the resulting diagram is the same.	
	
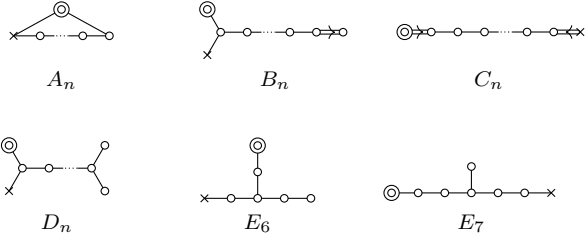
\begin{center}
	\vspace{-.25cm}
	 \begin{tabular}[320pt]{ccc}
 \parbox[t]{70pt}{\center  \dynkin[extended, affineMark=O]{A}{xo.oo}} & \parbox[t]{70pt}{\center \dynkin[extended, affineMark=O]{B}{xoo.ooo}}  & \parbox[t]{70pt}{ \vspace{2pt} \center \dynkin[extended, affineMark=O]{C}{ooo.oox}}  \\
			\footnotesize $A_n$ & \footnotesize $B_n$ & \footnotesize $C_n$ \\
		\end{tabular}
		
		\begin{tabular}[320pt]{ccccc}
		  \parbox[t]{70pt}{\center \dynkin[extended, affineMark=O]{D}{xoo.ooo}}  & 
			\parbox[t]{60pt}{\center \dynkin[extended, affineMark=O]{E}{xooooo}}  & \parbox[t]{80pt}{ \vspace{8pt} \center \dynkin[extended, affineMark=O]{E}{oooooox}}    \\
			\footnotesize $D_n$  & \footnotesize $E_6$ & \footnotesize $E_7$
		\end{tabular}
	\vspace{-.25cm}
	\end{center}

%	The choice of primitive roots of unity, means that the root system of $G_x$ contains the minimum number of roots possible. Since usual Dynkin subdiagrams are realizable as centralizers, we have 
%	

% (similarly to the last part of Proposition \ref{prop:subdiagram}).	
		
\end{proof}

Note that the diagram obtained by erasing a white node makes sense as a union of Satake diagrams. In this process, we can find diagrams of the type  $(G,G)$, $(G\times G,\Delta G)$, which are always regular by Remark \ref{ex:regular-extreme}.

% *** SAY PERHAPS MORE *** ??? \cite[\S 1.6]{springer-87}

\begin{example}\label{ex:sym-desc}
%	Some hints here for a direct proof? The neither pleasant nor nice pairs are  $(C_{2r},C_r+C_r)$, $(D_r,A_{r-1})$, $(E_7,E_6+\C)$ and some of the type $(BD,BD+BD)$.
%	%	$(D_{2n},D_r+D_s)$ with $r+s=2n$ and $r\neq s,s-1$, and $(D_{2n},B_r+B_s)$ with .
	
The pairs $(D_6,A_5+\C)$ and $(D_4,A_3+\C)$ are descendants of the pair $(E_7,D_6+\C)$:
$$				\dynkin[extended, affineMark=X]{E}{X*oo*o*}, \qquad \qquad 
		\qquad
		\dynkin[extended, affineMark=X]{E}{X*oo*XX}.$$
The pairs  $(E_7,E_6+\C)$ and $(D_8,D_6+D_2)$ are  descendants $(E_8,E_7+A_1)$:
$$
				\dynkin[extended, affineMark=X]{E}{o****ooo}, \qquad\qquad\qquad \dynkin[extended, affineMark=o]{E}{X****ooo}.
$$
We can also deduce that some pairs cannot be descendants of other pairs. For instance, the pair $(C_{2r},C_r+C_r)$ cannot be descendant of any exceptional symmetric pair. Since it has a double link, it could only appear as a descendant in the pairs with $G=G_2,F_4$. However, the colouring does not match, so it is not a possible descendant. 
\end{example}

Theorem \ref{theo:visual-descendants} combined with the classification of symmetric pairs gives the computation of all the descendants.

\begin{theorem}
	The list of descendants for simple complex symmetric pairs  is given by Tables \ref{tab:cla-des} and \ref{tab:exc-des}.
\end{theorem}

As for the notation, we use a bracket $\{$ or $\}$ when one out of several options must be chosen. These options may include $\varnothing$ meaning ``no pair'', as it will be the case when referring to a connected subdiagram formed only by white nodes, which can all be erased. For the sake of simplicity, in the classical cases we use the following notation:
\begin{itemize}
	\item $\sum (A,BD)$ for a sum of pairs of type $(A_{2t},B_t)$ or $(A_{2t-1},D_t)$, which can be possibly zero\footnote{The notation $BD$ has nothing to do with the notation of the non-reduced root system $BC$.}.
	\item $\sum (A,C)$ for a sum of pairs of type $(A_{2t-1},C_t)$ or $(A_1,A_1)$. 
	\item $\sum (A+A,A)$ for a sum or pairs of type $(A_t+A_t,A_t)$, which can be possibly zero. 
\end{itemize}
It is easily deduced by the erasing process what the constraints on the number of factors and their rank are. We sometimes use two names for the same Lie type, as $B_1$ or $C_1$, since it helps to interpret how it sits inside.

%
%\begin{remark}
%	More concretely, we have *** TO DISAPPEAR ***
%	 \begin{itemize}
%		\item  $\sum_{(t)} (A,BD)$ means $\sum_{i=1}^b (A_{2m_i},B_{m_i}) + \sum_{j=1}^d (A_{2n_j-1},B_{n_j}),$ for indices such that $\sum_{i=1}^b 2m_i + \sum_{j=1}^d (2m_j-1) \leq t+1-b-d$. 
%		\item $\sum_{(2b-1)} (A,C)$ means
%		$\sum_{i}(A_{2b_i-1},C_{b_i}) + \sum_l (A_1,A_1),$ for some indices such that $l + \sum_{i} b_i = b$.
%		\item $\sum_{(n)} (A+A,A)$ means $\sum_{i=1}^l (A_{m_i}+A_{m_i},A_{m_i})$ for $l-1 + \sum_{i=1}^l m_i \leq n$.
%	\end{itemize}
%\end{remark}

 \begin{remark}
 	Note that this recovers, over the complex numbers, the computation of descendants in \cite{ag-transactions} for types $(A,BD)$, and $(BD,BD+BD)$, and it gives extra information, as it says exactly how many factors of each Lie type may actually appear. To our knowledge, there were no previous computations of exceptional descendants.
 \end{remark}

\subsection{Some exceptional and Spin descendants}
\label{sec:exceptional-spin-descendants}

The result that we need for our applications in the next section is the following.
\begin{proposition}\label{prop:descendants-of-exceptional}
	The following are the only exceptional symmetric pairs where some neither pleasant nor nice pairs appear as descendants.
	\begin{itemize}
		\item Pairs of type $(D_4,B_{3})$ and $(D_5,D_4+\C)$ are  descendants of  $(E_6,D_5+\C)$.
		\item Pairs of type $(D_4,A_3+\C)$, $(D_5,B_3+B_1)$ and $(D_6,A_{5}+\C)$ are  descendants or parts of descendants of  $(E_7,D_6+A_1)$.
		\item Pairs of type $(E_7,E_6+\C)$ are  descendants of   $(E_7,E_6+\C)$.
		\item Pairs of type $(E_7,E_6+\C)$, $(D_7,B_5+B_1)$ and $(D_8,D_6+D_2)$  are  descendants of  $(E_8,E_7+A_1)$.
	\end{itemize}
\end{proposition}

\begin{proof}
	This is a combination of the classification of pleasant and nice pairs in Tables \ref{tab:pleasant-nice}, \ref{tab:pleasant-nice2} and \ref{tab:pleasant-nice3}, with the computation of descendants in Tables  \ref{tab:cla-des} and \ref{tab:exc-des}.  Some of them can be already seen in Example \ref{ex:sym-desc}.
\end{proof}

\begin{lemma}\label{lemma:Spin-4q+2}
	The only descendants of the pair $(\Spin_{4q+2},\Spin_{4q+1})$ are itself and $(\Spin_{4q},\Spin_{4q})$. 
\end{lemma}

\begin{proof}
	As in Table \ref{tab:cla-des}, this reads easily from its extended Dynkin diagram. $$\dynkin[extended]{D}{o**.***}$$
\end{proof}

			{\tiny \phantom{\dynkin{A}{IIIa}} } % NECESSARY TO AVOID ERROR BELOW FOR $(A_{r+s+1},A_r+A_s+\C)$\textbf{}
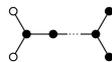
\vspace{-.5cm}

%\afterpage{
%{
	% \setlength{\textwidth}{\paperwidth}
	% \addtolength{\textwidth}{-3.15in}	
\begin{landscape}\clearpage
\begin{table}[!h]
	\begin{center}\footnotesize
		\vspace{-.6cm}
		\begin{tabular}{lcl}
						\hline 	
			symmetric pair & extended Satake diagram & descendants\\ \hline \hline
$(A_{2n},B_n)$, $(A_{2n-1},D_n)$ & \dynkin[extended]{A}{I}  & \parbox[c]{200pt}{ $\sum (A,BD)$\medskip} \\	 \hline
$(A_{2n-1},C_n)$    & \dynkin[extended]{A}{II} &  \parbox[c]{200pt}{\medskip $\sum (A,C)$ \medskip} \\ \hline
			\parbox[c]{90pt}{$(A_{r+s+1},A_r+A_s+\C)$} &  	\parbox[c]{75pt}{\vspace{-.05cm} \begin{center} \dynkin[extended]{A}{IIIa} \end{center} } & \parbox[c]{265pt}{ \medskip $\!\!\!\!\left.\!\begin{array}{c} (A_{2t+1},A_t+A_t+\C)\\ \varnothing \end{array}\!\!\right\}\! + \sum (A+A,A) + \left\{ \!\!\begin{array}{c} (A_{m+2u+1},A_{m+u}+A_u+\C) \\ (A_{m-1},A_{m-1}) \end{array}\!\!\right.$\! \\ with $m=|r-s|$ and $t,u\geq 0$.} \\ \hline % black nodes = max(|r-s|-1,0) 
$(C_{n},A_{n-1}+\C)$ & \dynkin[extended]{C}{I} & \parbox[c]{200pt}{\medskip $\left.\!\begin{array}{c} (C_{t},A_{t-1}+\C)\\ \varnothing \end{array}\!\!\right\}\! + \sum (A,BD) + \left\{ \!\!\begin{array}{c} (C_{u},A_{u-1}+\C)\\ \varnothing \end{array}\!\!\right. \! $ \medskip} \\ \hline
$(C_{r+s},C_r+C_s)$, $r\neq s$ & \dynkin[extended]{C}{IIa} & \parbox[c]{180pt}{\medskip $(C_{2t},C_t+C_t) + \sum (A,C) + (C_{m+2u},C_{m+u}+C_{u})$\\ with $m=|r-s|$  \medskip} \\ \hline
%$(C_{2r},C_r+C_r)$ & 	\dynkin[extended]{C}{IIb} & \parbox[c]{180pt}{\medskip $(C_{2t},C_t+C_t) + \sum (A,C) + (C_{2u},C_u+C_u)$ \medskip} \\		
%\hline
%\end{tabular}				
%\end{center}
%\end{table}
%
%% \sum \!\! \left\{\!\! \begin{array}{c} (A,B)\\
%%(A,D)\\	
%%\varnothing
%%\end{array}\!\! \right\} 
%
%%
%%\parbox[t]{75pt}{ $(B_{r+s},B_{r}+D_s)$,\\ $r=0$  }& \dynkin[extended]{B}{II}  &  \parbox[t]{245pt}{  $(B_{s-1},B_{s-1})$ }\\
%
%\begin{table}[h!]
%	\begin{center}\footnotesize
%		\begin{tabular}{lcl}
%			\hline
\parbox[c]{75pt}{	$(B_{r+s},B_{r}+D_s)$ }& \parbox[c]{75pt}{ \vspace{-.15cm} \begin{center} \dynkin[extended]{B}{I} \end{center} } & \parbox[c]{245pt}{ \medskip $\left.\!\!\begin{array}{c}
	(D_{2t},D_t+D_t)\\
	(D_{2t+2},B_t+B_t)\\	
	\varnothing
\end{array}\!\!\right\}\!+\! \sum (A,BD)\! +\! \left\{\!\!\begin{array}{c}
(B_{m+u},B_{m+\frac{u}{2}}+D_{\frac{u}{2}})\\ 
(B_{m+u},D_{m+\frac{u-1}{2}}+B_{\frac{u+1}{2}})\end{array}\!\!\right.$\\ with $m=max(r-s,s-r-1)$, for $u$ even or odd, respectively. \medskip} \\ \hline
\parbox[c]{75pt}{\bigskip $(D_{r+s},D_{r}+D_s)$,\\ $(D_{r+s+1},B_{r}+B_s)$,\\ $r\neq s-1$ }&
% \begin{minipage}[c]{90pt}
%\begin{center}
\parbox[c]{75pt}{\vspace{-.15cm} \begin{center} \dynkin[extended]{D}{Ia} \end{center} }
%\end{center}
%\end{minipage} 
&  \parbox[c]{245pt}{ \bigskip $\left.\!\!\begin{array}{c}
	(D_{2t},D_t+D_t)\\
	(D_{2t+2},B_t+B_t)\\	
	\varnothing
	\end{array}\!\!\right\}\!+\! \sum (A,BD)\! +\! \left\{\!\!\begin{array}{c}
	(D_{m+u},D_{m+\frac{u}{2}}+D_{\frac{u}{2}})\\
	(D_{m+u},B_{m+\frac{u-1}{2}}+B_{\frac{u-1}{2}})\end{array}\!\!\right.$\\ with $m=|r-s|$, for $u$ even or odd, respectively. \medskip}  \\ \hline
\parbox[c]{75pt}{	$(D_{r+s},D_{r}+D_s)$,\\ $(D_{r+s+1},B_{r}+B_s)$,\\ $r=s-1$ } & \parbox[c]{75pt}{\vspace{-.15cm} \begin{center} \dynkin[foldright]{D}[1]{ooo.oooo} \end{center} }&   \parbox[c]{245pt}{ \medskip $\left.\!\!\begin{array}{c}
	(D_{2t},D_t+D_t)\\
	(D_{2t+2},B_t+B_t)\\	
	\varnothing
	\end{array}\!\!\right\}\!+\! \sum (A,BD)\! +\! \left\{\!\!\begin{array}{c}
	(D_{2u-1},D_{u}+D_{u-1})\\
	(D_{2u},B_{u}+B_{u-1})\end{array}\!\!\right.$}  \\ \hline
\multirow{2}{*}{\parbox[t]{75pt}{\bigskip $(D_r, A_{r-1}+\C)$}} &  \dynkin[extended]{D}{IIIa} & \multirow{2}{*}{\parbox[t]{245pt}{ \medskip $\left.\!\!\begin{array}{c} (D_t,A_{t-1}+\C)\\ (A_1,A_1) \end{array}\!\!\right\}\! +\! \sum (A,C) +  \left\{ \!\!\begin{array}{c} (D_u,A_{u-1}+\C)\\ (A_1,A_1) \end{array}\!\!\right. $ \\ with $t$ even, and $u$ of the same parity as $r$.}  }\\
& \dynkin[foldright]{D}[1]{*o*.o*oo}  &   \\ 
\hline
\end{tabular}
\vspace{3pt}
\caption{Descendants of classical complex symmetric pairs}	\label{tab:cla-des}			
\end{center}
\end{table}
\vspace{0cm}
%\restoregeometry
\end{landscape}

%\begin{tikzpicture}
%\dynkin[extended]{D}{ooo.ooo}\begin{scope}[on background layer]
%\draw[/Dynkin diagram/foldStyle]($(root 5)$) -- ($(root 6)$);\end{scope}
%\end{tikzpicture}

{\tiny \phantom{\dynkin{E}{III}}} % NECESSARY TO AVOID ERROR BELOW 
\vspace{-2.2cm}

\begin{table}[!htbp]
\begin{center}\footnotesize
\begin{tabular}{lcl}\hline
	symmetric pair & extended Satake diagram & descendants (maximal before semicolon)\\ \hline \hline
$(G_2,A_1+A_1)$ & \dynkin[extended]{G}{I} & \parbox[c]{185pt}{\medskip $(A_2,B_1),(A_1,\C)+(A_1,\C)$; $(A_1,\C)$ \medskip} \\ \hline
  $(F_4,B_4)$  &  \dynkin[extended]{F}{I}  & \parbox[c]{185pt}{\medskip $(C_3,A_2+\C)+(A_1,\C)$, $(A_2,B_1)+(A_2,B_1)$, $(A_3,D_2)+(A_1,\C)$, $(B_4,B_2+D_2)$; \\
  	$(A_2,B_1)+(A_1,\C)$, $(B_3,D_2+B_1)$, $3\X(A_1,\C)$, $(B_2,B_1+\C)+(A_1,\C)$, and their summands\medskip} \\ \hline
  %   	
%  $\left. \begin{array}{c}
%  (A_2,B_1)\\	
%  (B_2,B_1+\C)
%  \end{array} \right\} + (A_1,\C) $, $(B_3,D_2+B_1)$\\
%  
 %  
$(F_4,C_3+A_1)$ & \dynkin[extended]{F}{II} & \parbox[c]{185pt}{\medskip$(B_4,B_1+D_3)$; $(B_3,B_3)$\medskip}  \\ \hline 
\begin{minipage}[c]{50pt} \vspace{0.2cm} $(E_{6},C_4)$ \end{minipage} & \begin{minipage}[c]{50pt} \vspace{0cm} \dynkin[extended]{E}{I}\end{minipage} &  \parbox[c]{185pt}{ \medskip $(A_5,D_3)+(A_1,\C)$, $3\X(A_2,B_1)$,   $(D_5,B_2+ B_2)$;\\ $\left.\!\!\begin{array}{c}
	(A_4,B_2)\\	
	2(A_2,B_1)
	\end{array}\right\} + (A_1,\C), $  $\left.\begin{array}{c}
	(A_3,D_2)\\	
	(A_2,B_1)
	\end{array}\right\} + 2(A_1,\C)$, \\
	$4\X (A_1,\C)$, and their summands\medskip } \\ \hline
% 
% $\sum (A_{r_i},BD)$ for $1\leq i\leq 3$ and $\sum r_i\leq 5$ with $(r_i)\neq (2,3)$
\begin{minipage}[c]{50pt} \vspace{0.4cm} $(E_6,A_5+A_1)$ \end{minipage} & \begin{minipage}[c]{50pt} \vspace{0.2cm} \dynkin[extended]{E}{II} \end{minipage} & 
\parbox[c]{185pt}{ \vspace{0.2cm} $(D_5,D_2+D_3)$, $(A_1+A_1,A_1)+(A_3,D_2)$, $(A_2+A_2,A_2)+(A_2,B_1)$, $(A_5,A_2+A_2+\C)+(A_1,\C)$; $(A_1+A_1,A_1)+(A_2,B_1)$, $(A_3,A_1+A_1+\C)+(A_1,\C)$, $(D_4,B_2+B_1)$, $(A_1+A_1,A_1)+ 2\X(A_1,\C)$, $(A_2+A_2,A_2)+(A_1,\C)$, and their summands\vspace{.1cm}  }\\ \hline
\begin{minipage}[c]{50pt} \medskip $(E_6,F_4)$ \medskip\end{minipage} &  \begin{minipage}[c]{50pt} \medskip \dynkin[extended]{E}{IV} \medskip \end{minipage} &  \parbox[c]{185pt}{  $(D_5,B_4)$; $(D_4,D_4)$ }\\  \hline 
\begin{minipage}[c]{50pt} \medskip $(E_6,D_5+\C)$ \medskip \end{minipage} & \begin{minipage}[c]{50pt} \medskip \dynkin[extended]{E}{III} \medskip \end{minipage} & \parbox[c]{185pt}{ \medskip $(D_5,D_4+\C)$, $(A_5, A_4+A_1+\C)+(A_1,\C)$; $(D_4,B_3)$, $(A_3,A_3)+(A_1,\C)$,  $(A_5, A_4+A_1+\C)$, $(B_3,B_3)$ \medskip }\\  \hline 
\begin{minipage}[c]{50pt} \medskip $(E_{7},A_7)$ \medskip \end{minipage} & \begin{minipage}[c]{50pt}  \dynkin[extended]{E}{V}\medskip \end{minipage}& \parbox[c]{185pt}{ \medskip$(D_6,D_3+D_3)+(A_1,\C)$,\!\!\! $(A_5,D_3)+(A_2,B_1)$, $2\X(A_3,D_2) + (A_1,\C)$, $(E_6,C_4)$; $3\X(A_2,B_1)$, 
	$\left.\!\!\begin{array}{c}
	(D_5,B_2+B_2), (A_5,D_3)\\	
	(A_3,D_2)+(A_2,B_1), 2(A_2,B_1)
	\end{array}\right\} + (A_1,\C), $
	$(D_4,B_2+B_2)+2(A_1,\C)$, $(A_4,B_2)+(A_2,B_1)$, $(A_3,D_2)+3(A_1,\C)$,  $(A_2,B_1)+3\X(A_1,\C)$, $5\X(A_1,\C)$, and their summands.\medskip}  \\  \hline
% $\sum (A_{r_i},BD)$ for $1\leq i\leq 4$ and $\sum r_i\leq 6$ with $(r_i)\neq (1,1,4)$,  
$(E_7,D_6+A_1)$ & \dynkin[extended]{E}{VI} & \parbox[c]{185pt}{\medskip $(D_6,A_5+\C)+(A_1,\C)$, $(A_5,C_3)+(A_2,B_1)$, $(A_3,D_2)+(A_3,C_2)+(A_1,A_1)$, $(D_6,D_4+D_2)+(A_1,A_1)$;  $(A_5,C_3)+(A_1,\C)$, $3\X (A_1,A_1)$,\\
	 $\left.\!\!\!\!\!\begin{array}{c}
	(D_5,B_3+B_1), (D_4,A_3+\C), \\	
	(A_3,C_2) + \left\{\!\!\begin{array}{c}
	(A_3,D_2),  (A_2,B_1), \\	
	2(A_1,\C),\!(A_1,\C),\!\varnothing
	\end{array}\!\!\right\}
	\end{array}\!\!\!\!\right\}\!+\!(A_1,A_1),$\medskip }  \\ \hline
$(E_7,E_6+\C)$ & \dynkin[extended]{E}{VII} & \parbox[c]{185pt}{\medskip $(D_6,D_5+D_1) + (A_1,\C)$, $(E_6,F_4)$;  $(D_5,B_4)+ \left\{ (A_1,\C), \varnothing \right\}$, $(D_4,D_4)+ \left\{2(A_1,\C), (A_1,\C), \varnothing \right\}$, $(D_6,D_5+D_1)$} \\ \hline
$(E_8,D_8)$ & \dynkin[extended]{E}{VIII} & \parbox[c]{185pt}{ \medskip $(D_8,D_4+D_4)$, $(A_7,D_4)+(A_1,\C)$, $(A_8,B_4)$, $(A_5,B_3)+(A_2,B_1)+(A_1,\C)$,  $2\X (A_4,B_2)$, $(D_5,B_2+B_2)+(A_3,D_2)$,  $(E_6,C_4)+(A_2,B_1)$, $(E_7,A_7)+(A_1,\C)$; $(E_6,C_4)+(A_1,\C)$, \\
	$\left. \begin{array}{c}
	(D_5,B_2+B_2)\\	
	(D_4,D_2+D_2)
	\end{array} \right\} + \left\{ \begin{array}{c}
	(A_3,D_2), \; (A_2,B_1), \\	
	2\X(A_1,\C), \; (A_1,\C), \; \varnothing
	\end{array} \right. $\\
	$(A_4,B_2)+(A_3,D_2)$, $(A_4,B_2)+(A_2,B_1)+(A_1,\C)$, $2\X(A_3,D_2)+(A_1,\C)$, $3\X(A_2,B_1)+(A_1,\C)$, $2\X(A_2,B_1)+2\X(A_1,\C)$, $(A_2,B_1)+4\X(A_1,\C)$, $5\X(A_1,\C)$, and their summands\medskip}\\ \hline
%  $\sum (A_{r_i},BD)$ for $1\leq i\leq 5$ and $\sum r_i\leq 6$, and their summands.
$(E_8,E_7+A_1)$ & \dynkin[extended]{E}{IX} &  \parbox[c]{185pt}{ \medskip $(D_8,D_6+D_2)$, $(D_5,B_4)+(A_3,D_2)$,\\ $(E_6,F_4)+(A_2,B_1)$, $(E_7,E_6+\C)+(A_1,\C)$, $(D_7,B_5+B_1)$, $(D_6,D_5+D_1) + \big\{ (A_1,\C), \varnothing \big\}$; \\
$\left. \begin{array}{c}
	(D_5,B_4)\\	
	(D_4,D_4)
	\end{array} \right\} + \left\{ \begin{array}{c}
	(A_3,D_2), \; (A_2,B_1), \\	
	2\X(A_1,\C), \; (A_1,\C), \; \varnothing
\end{array} \right. $} \\ \hline
\end{tabular}				
\vspace{3pt}
\caption{Descendants of exceptional complex symmetric pairs}\label{tab:exc-des}
\end{center}
\end{table}
% }
%}

\section{Exceptional and Spin Gelfand pairs, and reduction of conjectures}
\label{sec:gelfand-pairs-conjectures}

We conclude by combining the results of Sections \ref{sec:pleasant} and \ref{sec:visual} and proving the Gelfand property for eight out of the twelve exceptional complex symmetric pairs, and for an infinite family of $\Spin$ pairs. 

\begin{theorem}\label{theo:8-Gelfand}
	The complex symmetric pairs $(G_2,A_1+A_1)$, $(F_4,B_4)$, $(F_4,C_3+A_1)$, $(E_6,C_4)$, $(E_6,A_5+A_1)$, $(E_7,A_7)$ and $(E_8,D_8)$, together with the infinite family $(\Spin_{4q+2},\Spin_{4q+1})$, are Gelfand pairs.
\end{theorem}

\begin{proof}
	From Proposition \ref{prop:descendants-of-exceptional}, all exceptional symmetric pairs but $(E_6,F_4)$ and $(E_6,D_5+\C)$, $(E_7,E_6+\C)$, $(E_7,D_6+A_1)$ and $(E_8,E_7+A_1)$ are regular and have pleasant and/or nice (hence regular) descendants. The same holds for the pair $(\Spin_{4q+2},\Spin_{4q+1})$ because of Lemma \ref{lemma:Spin-4q+2}. By Proposition \ref{prop:criterion-complex} we have that they are Gelfand pairs.
\end{proof}

Moreover, we can reduce the proof of the Gelfand property for the remaining four exceptional symmetric pairs to a statement about the regularity of one exceptional and eight classical pairs. Combining Proposition \ref{prop:descendants-of-exceptional} with Proposition \ref{prop:criterion-complex} and Lemma \ref{lem:quot-sym-pair} gives the following.

\begin{proposition}\label{prop:exceptional-Gelfand}
	All exceptional complex symmetric pairs are Gelfand pairs if $(D_4,B_3)$, $(D_4,A_3+\C)$,  $(D_5,D_4+\C)$, $(D_6,A_{5}+\C)$, $(D_7,B_5+B_1)$,  $(E_7,E_6+\C)$ and $(D_8,D_6+D_2)$ are regular. 
\end{proposition}

Finally, we give a sufficient condition for van Dijk's conjecture and reduce Aizenbud-Gourevitch conjecture to the regularity of one exceptional pair and some classical families. Combining Tables \ref{tab:pleasant-nice}, \ref{tab:pleasant-nice2}, \ref{tab:pleasant-nice3}, \ref{tab:cla-des} and \ref{tab:exc-des} with  Proposition \ref{prop:criterion-complex}, Lemma \ref{lem:quot-sym-pair}, Theorem \ref{theo:8-Gelfand} and Proposition \ref{prop:exceptional-Gelfand} gives the following.

\begin{proposition}\label{prop:conjectures}
All complex symmetric pairs are regular (Aizenbud-Gourevitch conjecture) and Gelfand pairs (implying van Dijk's conjecture) if the families of pairs $(D_r,A_{r-1}+\C)$, $(C_{2r},C_r+C_r)$, the families $(\Spin_{r+s},\Spin_r \times_{\Z_2} \Spin_s)$ for $|r-s|>2$, and $r,s$ even if $r+s\neq 4q+2$, and the pair $(E_7,E_6+\C)$ are regular. 
\end{proposition}

% Our work highlights the importance of looking also at $\Spin$ pairs and not just orthogonal ones, and reduces the proof of the Gelfand property for exceptional complex symmetric pairs to the proof of regularity of just one exceptional pair and some low-rank classical ones (Proposition \ref{prop:exceptional-Gelfand}).

%%% The regularity of the pairs in Proposition \ref{prop:exceptional-Gelfand} and eventually the families in Proposition \ref{prop:conjectures} is joint work in progress \cite{carmeli-rubio} by means of distributional methods.

% \section{Add something about non-connected case in general???}
% Extension of group by automorphisms...

%\newpage
%\appendix

\bibliographystyle{alpha}
\bibliography{bib-gelfand}

\begin{thebibliography}{{Yok}09}

\bibitem[Ada14]{adams-14}
Jeffrey Adams.
\newblock The real {C}hevalley involution.
\newblock {\em Compos. Math.}, 150(12):2127--2142, 2014.

\bibitem[AG09]{ag-duke}
Avraham Aizenbud and Dmitry Gourevitch.
\newblock Generalized {H}arish-{C}handra descent, {G}elfand pairs, and an
  {A}rchimedean analog of {J}acquet-{R}allis's theorem.
\newblock {\em Duke Math. J.}, 149(3):509--567, 2009.
\newblock With an appendix by the authors and Eitan Sayag.

\bibitem[AG10]{ag-transactions}
Avraham Aizenbud and Dmitry Gourevitch.
\newblock Some regular symmetric pairs.
\newblock {\em Trans. Amer. Math. Soc.}, 362(7):3757--3777, 2010.

\bibitem[AGS08]{ags-compositio}
Avraham Aizenbud, Dmitry Gourevitch, and Eitan Sayag.
\newblock {$({\rm GL}_{n+1}(F),{\rm GL}_n(F))$} is a {G}elfand pair for any
  local field {$F$}.
\newblock {\em Compos. Math.}, 144(6):1504--1524, 2008.

\bibitem[AGS09]{ags-09}
Avraham Aizenbud, Dmitry Gourevitch, and Eitan Sayag.
\newblock {$({\rm O}(V\oplus F),{\rm O}(V))$} is a {G}elfand pair for any
  quadratic space {$V$} over a local field {$F$}.
\newblock {\em Math. Z.}, 261(2):239--244, 2009.

\bibitem[Aiz13]{aizenbud-13}
Avraham Aizenbud.
\newblock A partial analog of the integrability theorem for distributions on
  {$p$}-adic spaces and applications.
\newblock {\em Israel J. Math.}, 193(1):233--262, 2013.

\bibitem[Ara62]{araki}
Sh\^or\^o Araki.
\newblock On root systems and an infinitesimal classification of irreducible
  symmetric spaces.
\newblock {\em J. Math. Osaka City Univ.}, 13:1--34, 1962.

\bibitem[AV92]{adams-vogan-92}
Jeffrey Adams and David~A. Vogan, Jr.
\newblock Harish-{C}handra's method of descent.
\newblock {\em Amer. J. Math.}, 114(6):1243--1255, 1992.

\bibitem[AV16]{adams-vogan}
Jeffrey Adams and David~A. Vogan, Jr.
\newblock Contragredient representations and characterizing the local
  {L}anglands correspondence.
\newblock {\em Amer. J. Math.}, 138(3):657--682, 2016.

\bibitem[AvD06]{aparicio-vanDijk}
Sofía Aparicio and Gerrit van Dijk.
\newblock Complex generalized gelfand pairs.
\newblock {\em Tambov University Reports, Series: Natural and Technical
  Sciences}, 11(1):6--12, 2006.

\bibitem[Bor61]{borel-61}
Armand Borel.
\newblock Sous-groupes commutatifs et torsion des groupes de {L}ie compacts
  connexes.
\newblock {\em T\^{o}hoku Math. J. (2)}, 13:216--240, 1961.

\bibitem[Bou68]{bourbaki456}
Nicolas Bourbaki.
\newblock {\em \'{E}l\'{e}ments de math\'{e}matique. {F}asc. {XXXIV}. {G}roupes
  et alg\`ebres de {L}ie. {C}hapitre {IV}: {G}roupes de {C}oxeter et syst\`emes
  de {T}its. {C}hapitre {V}: {G}roupes engendr\'{e}s par des r\'{e}flexions.
  {C}hapitre {VI}: syst\`emes de racines}.
\newblock Actualit\'{e}s Scientifiques et Industrielles, No. 1337. Hermann,
  Paris, 1968.

\bibitem[Bum13]{bump}
Daniel Bump.
\newblock {\em Lie groups}, volume 225 of {\em Graduate Texts in Mathematics}.
\newblock Springer, New York, second edition, 2013.

\bibitem[{Car}15]{carmeli2015}
Shachar {Carmeli}.
\newblock {On the Stability and Gelfand Property of Symmetric Pairs}.
\newblock {\em arXiv e-prints}, arXiv:1511.01381, November 2015.

\bibitem[Cas89]{casselman-89}
W.~Casselman.
\newblock Canonical extensions of {H}arish-{C}handra modules to representations
  of {$G$}.
\newblock {\em Canad. J. Math.}, 41(3):385--438, 1989.

\bibitem[Die75]{dieudonne}
J.~Dieudonn\'{e}.
\newblock {\em \'{E}l\'{e}ments d'analyse. {T}ome {VI}. {C}hapitre {XXII}}.
\newblock Gauthier-Villars \'{E}diteur, Paris, 1975.
\newblock Cahiers Scientifiques, Fasc. XXXIX.

\bibitem[Dyn00]{dynkin}
E.~B. Dynkin.
\newblock Semisimple subalgebras of semisimple {L}ie algebras.
\newblock In {\em Selected papers of {E}. {B}. {D}ynkin with commentary}, pages
  175--308. American Mathematical Society, Providence, RI; International Press,
  Cambridge, MA, 2000.

\bibitem[GK75]{gelfand-kazhdan}
I.~M. Gelfand and D.~A. Kajdan.
\newblock Representations of the group {${\rm GL}(n,K)$} where {$K$} is a local
  field.
\newblock In {\em Lie groups and their representations ({P}roc. {S}ummer
  {S}chool, {B}olyai {J}\'{a}nos {M}ath. {S}oc., {B}udapest, 1971)}, pages
  95--118. Halsted, New York, 1975.

\bibitem[GL83]{gorenstein-lyons}
Daniel Gorenstein and Richard Lyons.
\newblock The local structure of finite groups of characteristic {$2$} type.
\newblock {\em Mem. Amer. Math. Soc.}, 42(276):vii+731, 1983.

\bibitem[Gro91]{gross}
Benedict~H. Gross.
\newblock Some applications of {G}el\cprime fand pairs to number theory.
\newblock {\em Bull. Amer. Math. Soc. (N.S.)}, 24(2):277--301, 1991.

\bibitem[Hum95]{humphreys}
James~E. Humphreys.
\newblock {\em Conjugacy classes in semisimple algebraic groups}, volume~43 of
  {\em Mathematical Surveys and Monographs}.
\newblock American Mathematical Society, Providence, RI, 1995.

\bibitem[LS99]{levasseur-stafford}
T.~Levasseur and J.~T. Stafford.
\newblock Invariant differential operators on the tangent space of some
  symmetric spaces.
\newblock {\em Ann. Inst. Fourier (Grenoble)}, 49(6):1711--1741, 1999.

\bibitem[Ric82]{richardson}
R.~W. Richardson.
\newblock Orbits, invariants, and representations associated to involutions of
  reductive groups.
\newblock {\em Invent. Math.}, 66(2):287--312, 1982.

\bibitem[Sat60]{satake}
Ichir\^{o} Satake.
\newblock On representations and compactifications of symmetric {R}iemannian
  spaces.
\newblock {\em Ann. of Math. (2)}, 71:77--110, 1960.

\bibitem[{Say}08]{sayag}
Eitan {Sayag}.
\newblock {(GL(2n,C),SP(2n,C)) is a Gelfand Pair}.
\newblock {\em arXiv e-prints}, arXiv:0805.2625, May 2008.

\bibitem[Sek85]{sekiguchi}
Jir\=o Sekiguchi.
\newblock Invariant spherical hyperfunctions on the tangent space of a
  symmetric space.
\newblock In {\em Algebraic groups and related topics ({K}yoto/{N}agoya,
  1983)}, volume~6 of {\em Adv. Stud. Pure Math.}, pages 83--126.
  North-Holland, Amsterdam, 1985.

\bibitem[SS70]{springer-steinberg}
T.~A. Springer and R.~Steinberg.
\newblock Conjugacy classes.
\newblock In {\em Seminar on {A}lgebraic {G}roups and {R}elated {F}inite
  {G}roups ({T}he {I}nstitute for {A}dvanced {S}tudy, {P}rinceton, {N}.{J}.,
  1968/69)}, Lecture Notes in Mathematics, Vol. 131, pages 167--266. Springer,
  Berlin, 1970.

\bibitem[SZ11]{sun-zhu}
Binyong Sun and Chen-Bo Zhu.
\newblock A general form of {G}elfand-{K}azhdan criterion.
\newblock {\em Manuscripta Math.}, 136(1-2):185--197, 2011.

\bibitem[Tho84]{thomas}
Erik G.~F. Thomas.
\newblock The theorem of {B}ochner-{S}chwartz-{G}odement for generalised
  {G}el\cprime fand pairs.
\newblock In {\em Functional analysis: surveys and recent results, {III}
  ({P}aderborn, 1983)}, volume~90 of {\em North-Holland Math. Stud.}, pages
  291--304. North-Holland, Amsterdam, 1984.

\bibitem[vD86]{vanDijk-86}
G.~van Dijk.
\newblock On a class of generalized {G}el\cprime fand pairs.
\newblock {\em Math. Z.}, 193(4):581--593, 1986.

\bibitem[vD08]{vanDijk-08}
G.~van Dijk.
\newblock {\em Gelfand pairs and beyond}, volume~11 of {\em COE Lecture Note}.
\newblock Kyushu University, Faculty of Mathematics, Fukuoka, 2008.
\newblock Math-for-Industry (MI) Lecture Note Series.

\bibitem[Vus74]{vust}
T.~Vust.
\newblock Op\'eration de groupes r\'eductifs dans un type de c\^ones presque
  homog\`enes.
\newblock {\em Bull. Soc. Math. France}, 102:317--333, 1974.

\bibitem[Wal88]{wallach}
Nolan~R. Wallach.
\newblock {\em Real reductive groups. {I}}, volume 132 of {\em Pure and Applied
  Mathematics}.
\newblock Academic Press, Inc., Boston, MA, 1988.

\bibitem[{Yok}09]{yokota}
Ichiro {Yokota}.
\newblock {Exceptional Lie groups}.
\newblock {\em ArXiv e-print arXiv:0902.0431}, February 2009.

\end{thebibliography}

\end{document}